\documentclass[11pt]{article}

\usepackage{amsmath, amsfonts, amssymb}
\usepackage{mathrsfs}
\usepackage{theorem}
\usepackage{bm}

\newtheorem{thm}{Theorem}[section]

\newtheorem{lem}[thm]{Lemma}

{\theorembodyfont{\rmfamily} \newtheorem{rem}[thm]{Remark}}
{\theorembodyfont{\rmfamily} }
\newcommand{\norm}[1]{\lVert#1\rVert}

\newcommand{\ra}{\rightarrow}

\def\R{\mathbb R}

\def\C{\mathscr C}

\def\F{\mathscr F}
\def\d{\text{\rm{d}}}
\def\E{\mathbb E}
\def\p{\mathbb P}
\def\e{\text{\rm{e}}}

\def\la{\langle}
\def\raa{\rangle}
\def\La{\Lambda}

\def\veps{\varepsilon}

\def\de{\delta}

\def\S{\mathcal S}
\def\C{\mathscr C}

\def\wt{\widetilde}
\def\spec{\mathrm{Spec}}
\def\Re{\mathrm{Re}}

\newtheorem{Theorem}{Theorem}[section]

\newtheorem{Example}[Theorem]{Example}

\newtheorem{Assumption}[Theorem]{Assumption}
\newcommand{\beq}[1]{\begin{equation} \label{#1}}
\newcommand{\eeq}{\end{equation}}
\newcommand{\bed}{\begin{displaymath}}
\newcommand{\eed}{\end{displaymath}}
\newcommand{\bea}{\bed\begin{array}{rl}}
\newcommand{\eea}{\end{array}\eed}
\newcommand{\disp}{\displaystyle}
\newcommand{\ad}{&\!\!\!\disp}

\newcommand{\barray}{\begin{array}{ll}}
\newcommand{\earray}{\end{array}}

\newcommand{\wdt}{\widetilde}
\newcommand{\wdh}{\widehat}

\newcommand{\sg}{\sigma}
\newcommand{\one}{\mathbf{1}}
\newcommand{\B}{{\mathcal B}}
\newcommand{\m}{\mathfrak m}
\def\jinghai#1 {\fbox {\footnote {\ }}\ \footnotetext {From Jinghai: #1}}
\def\fubao#1 {\fbox {\footnote {\ }}\ \footnotetext {From Fubao: #1}}

\newenvironment{proof}{{\noindent\it Proof.}\ }{\hfill $\square$\par}

\numberwithin{equation}{section}

\allowdisplaybreaks

\begin{document}

\title{Stabilization of regime-switching processes by feedback control based on discrete time observations II: state-dependent case\footnote{Supported in
 part by the National Natural Science Foundation of China (Grant Nos. 11671034, 11771327, 11431014)}}

\author{Jinghai Shao\thanks{Center for Applied Mathematics, Tianjin University, Tianjin 300072, China. Email: shaojh@tju.edu.cn.} \and Fubao Xi\thanks{School of Mathematics and Statistics, Beijing Institute of Technology, Beijing 100081,
China. Email: xifb@bit.edu.cn.}}

\maketitle

\begin{abstract}
  This work investigates the almost sure stabilization of a class of regime-switching systems based on discrete-time observations of both continuous and discrete components. It develops Shao's work [SIAM J. Control Optim., 55(2017), pp. 724--740] in two aspects: first, to provide sufficient conditions for almost sure stability in lieu of moment stability; second, to investigate a class of state-dependent regime-switching processes instead of state-independent ones. To realize these developments, we establish an estimation of the exponential functional of Markov chains based on the spectral theory of linear operator. Moreover, through constructing order-preserving coupling processes based on Skorokhod's representation of jumping process, we realize the control from up and below of the evolution of state-dependent switching process by state-independent Markov chains. 
\end{abstract}



\section{Introduction}\label{Intro}

This work is concerned with the stability of the following regime-switching process:
\begin{equation}\label{1.1}
\d X(t)=\big[a(X(t),\La(t))-b(\La(\de(t)))X(\de(t))\big]\d t+\sigma(X(t),\La(t))\d W(t),
\end{equation}
where $\de(t)=[t/\tau]\tau$, $[t/\tau]$ denotes the integer part of the number $t/\tau$, $\tau$ is a positive constant, and $(W(t))$ is a $d$-dimensional Wiener process. Here $(\La(t))$ is a continuous time jumping process on $\S=\{1,2,\ldots,M\}$, $M<+\infty$, satisfying
\begin{equation}\label{1.2}
\p(\La(t+\Delta)=j|\La(t)=i,\,X(t)=x)=\left\{\begin{array}{ll} q_{ij}(x)\Delta+o(\Delta), & \text{if}\ i\neq j,\\
1+q_{ii}(x)\Delta+o(\Delta), &\text{if}\ i=j,
\end{array}\right.
\end{equation} provided $\Delta \downarrow 0$, where
$0\le q_{ij}(x)<+\infty$ for all $i, j \in \S$ with $i \neq j$. As
usual, we assume that for each $x\in\R^d$, the $Q$-matrix $Q_{x}=\big(q_{ij}(x)\bigr)$ is conservative; namely $q_{i}(x):=-q_{ii}(x)=\sum_{j\in \S\setminus\{i\}}q_{ij}(x)$ for all $i\in\S$.
Equation \eqref{1.1} is a type of stochastic functional differential equation. In current work we shall provide sufficient conditions to ensure the almost sure stability of the system \eqref{1.1} and \eqref{1.2}.

Regime-switching processes have drawn much attention due to the demand of modeling, analysis and computation of complex dynamical systems, and have been widely used in mathematical finance, engineer, biology etc (see, e.g. the monographs \cite{MY,YZ}). Compared with the classical stochastic processes without switching, regime-switching processes can reflect the random change of the environment in which the concerning system lived. Then there are many new difficulties and phenomena appeared in the study of regime-switching processes. See, for instance, \cite{BBG1999,KZY2007,MY,SX,YX2010,YZ} and references therein on the stability of such system; \cite{BL,CH,PP,SX,Sh15a,Sh15b} on the recurrence of such system; \cite{SY,Bar,HS18} on the heavy or light tail behavior of the invariant probability measure of such system. Besides, there are some literature on the regime-switching processes driven by L\'evy processes \cite{CWZZ,Xi,XZ17,YX2010}. Recently, there are also some studies on regime-switching stochastic functional differential equations, e.g. \cite{BSY,M13,Mao14,Shao17a,Mao15,NY}.

Our motivation to study the equation \eqref{1.1} is to stabilize an unstable system \eqref{1.1} with $b=0$ based on discrete time observations of $(X(t))$ and $(\La(t))$.  Such stabilization problem for regime-switching processes was first raised by Mao \cite{M13} for the sake of saving cost and being more realistic. There, Mao investigated the mean-square stability of the following controlled system:
\[\d X(t)=(a(X(t),\La(t))-b(X(\de(t),\La(t))))\d t+\sigma(X(t),\La(t))\d W(t),\]
where $(\La(t))$ is a continuous time Markov chain independent of the Wiener process $(W(t))$. Subsequently, many works were devoted to developing this stabilization problem. See, for example, \cite{Mao14,Mao15}. Especially, in \cite{Mao15},  some sufficient conditions were provided to ensure this system to be almost surely stable. Inspired by these works, \cite{Shao17a} investigated the stability of such kind of system not only based on discrete time observations of $(X(t))$ but also according to discrete time observations  of $(\La(t))$. This needs to overcome the essential difference between the path property of $(X(t))$ and $(\La(t))$. For the continuous process $(X(t))$, since $X(t-\tau)$ tends to $X(t)$ as $\tau\ra0$, the difference between $X(t)$ and $X(t-\tau)$ can be controlled when $\tau$ is sufficiently small.  However, for the jumping process $(\La(t))$, even $\La(t)$ and $\La(t-):=\lim_{s\uparrow t} \La(s)$ may be quite different. Therefore, in \cite{Shao17a} Shao takes advantage of the independence of $(W(t))$ and $(\La(t))$ to control the evolution of $(X(t))$ through the long time behavior of  $(\La(t))$ and $(\La(n\tau))_{n\geq 0}$. Precisely, it  was shown that
\begin{equation}\label{co-1}
\E |X(t)|^2\leq |X(0)|^2\E\Big[\e^{\int_0^t f(\La(r))+K_\tau g(\La(\de(r)))\d r}\Big],
\end{equation}
where $f,\,g:\S\ra\R$, $K_\tau$ is a constant related to $\tau$. As an embedded Markov chain, $(\La(n\tau))_{n\geq 0}$ has the same stationary distribution as that of $(\La(t))$. However, as mentioned in \cite[Remark 3.3]{Shao17a}, the following kind of quantity cannot be handled at that time in order to show the almost sure stability
\[\E\int_0^\infty\e^{\int_0^t g(\La(\de(s)))\d s}\d t<\infty,\]
which could be dealt with in current work under the help of spectral theory of linear operator (see Lemma \ref{t-2.1} below).

In this work, we will overcome two difficulties to establish the  almost sure stability of $(X(t),\La(t))$ given by \eqref{1.1} and \eqref{1.2}. First, via the spectral theory of linear operators, we provide estimates from upper and below the exponential functional of  Markov chain $(Y_n)_{n\geq 0}$. Second, through using Skorokhod's representation of jumping processes or constructing order-preserving
coupling, we can control the evolution of state-dependent jumping process $(\La(t))$ by some auxiliary state-independent Markov chains. Moreover, to ensure the existence of the system $(X(t),\La(t))$ satisfying \eqref{1.1} and \eqref{1.2}, and the existence of order-preserving couplings, a general result on the existence of regime-switching stochastic functional differential equation is established. In particular, we only assume that  $x\mapsto q_{ij}(x)$ is continuous and $q_i(x)$ is of polynomial order of growth for $i,\,j\in\S$ (see Theorem \ref{exiuni} below for details).

The remainder of this paper is arranged as follows: Section 2 presents some necessary preparation results concerning the existence of solution for state-dependent regime-switching stochastic functional differential equations, estimate of exponential functional of Markov chains, and constructing auxiliary Markov chains to control the evolution of state-dependent jumping process $(\La(t))$. Section \ref{Almost} studies the almost sure stability for a class of regime-switching systems based on discrete-time observations. Using the technique used in \cite{Mao15}, we can prove  our main result, Theorem \ref{t3.4}, of this work. Finally, the proof of the existence and uniqueness of solution for regime-switching stochastic functional differential equations is appended in  Appendix A.

\section{Preliminary results}\label{Pre}

Let us begin this section with the existence and uniqueness of above system \eqref{1.1} and \eqref{1.2} which can be viewed as a regime-switching stochastic functional differential equations (SFDEs). Here we collect the conditions used in this work on the coefficients of \eqref{1.1} and transition rate matrix $(q_{ij}(x))$.
We can consider a little more general SFDE:
\begin{equation}\label{1.1b}
\d X(t)=\big[a(X(t),\La(t))-b(X(\de(t)),\La(\de(t)))\big]\d t+\sigma(X(t),\La(t))\d W(t).
\end{equation}
Suppose the coefficients $a(\cdot,\cdot):\R^d\times\S\ra \R^d$, $b(\cdot,\cdot):\R^d\times\S\ra [0,\infty)$ and $\sigma(\cdot,\cdot):\R^d\times \S\ra \R^{d\times d}$ satisfy the following conditions.
\begin{itemize}
  \item[(\textbf{H1})]\ There exist nonnegative functions $C(\cdot)$ and $c(\cdot)$ on $\S$ and a positive constant $\hat b$ such that
      \[c(i)|x|^2\leq 2\la a(x,i),x\raa+\|\sigma(x,i)\|_{\textrm{HS}}^2\leq C(i)|x|^2,\ \ (x,i)\in \R^d\times \S,\]
      \[|b(x,i)|\le \hat b(1+|x|), \ x\in \R^d,\ \ (x,i)\in \R^d\times \S,\]
      where $\|\sigma(x,i)\|_{\textrm{HS}}^2=\mathrm{trace}(\sigma\sigma^\ast)(x,i)$ with $\sigma^\ast$ denoting the transpose of the matrix $\sigma$.
  \item[(\textbf{H2})]\ There exists a positive constant $\bar K$ such that
   \[ |a(x,i)-a(y,i)|+|b(x,i)-b(y,i)|+\|\sigma(x,i)-\sigma(y,i)\|_{\textrm{HS}}\leq \bar K|x-y|, \ x,y\in \R^d, \ i\in \S.\]
   \end{itemize}
Moreover, let the $Q$-matrix $Q_{x}=\big(q_{ij}(x)\bigr)$ satisfy the following conditions:
\begin{itemize}
  \item[$\mathrm{({\bf Q1})}$] $x\mapsto q_{ij}(x)$ is continuous for every $i,\,j\in\S$.
  \item[$\mathrm{({\bf Q2})}$] $H:=\sup_{x\in\R^d}\max_{i\in \S} q_i(x)<\infty$.
\end{itemize}
The condition (\textbf{Q2}) is used in the control of the evolution of $(\La(t))$ through Markov chains. If only for the aim of existence and uniqueness of the dynamical system $(X(t),\La(t))$, we can use a weaker condition as follows:
\begin{itemize}
  \item[$\mathrm{({\bf Q2'})}$] $q_i(x)=\sum_{j\in {\S\setminus{\{i\}}}}q_{ij}(x)\le K_{0}(1+|x|^{\kappa_{0}})$ for every $(x,i)\in \R^d\times \S$, where $K_0$ and $\kappa_0$ are positive constants.
\end{itemize}

\begin{thm}\label{exiuni}
Assume conditions (\textbf{H1}), (\textbf{H2}), (\textbf{Q1}) and (\textbf{Q2'}) hold. Then there exists a unique nonexplosive solution $(X,\La)$ to   \eqref{1.1b} and \eqref{1.2}.
\end{thm}

In order to preserve the flow of presentation, we relegate the proof of Theorem \ref{exiuni} to Appendix A. We provide a very explicit construction of the solution $(X,\La)$ to regime-switching SFDE \eqref{1.1b} and \eqref{1.2} and prove the nonexplosiveness of the solution. Compared with the corresponding results in \cite{Sh15,Shao17a} where $x\mapsto q_{ij}(x)$ is assumed to be Lipschitzian, here we only suppose $x\mapsto q_{ij}(x)$ to be continuous. This greatly simplifies the conditions to be verified so that the coupling process constructed below exists. On the other hand, contrary to the usual boundedness assumption (\textbf{Q2}) imposed in the previous works such as \cite{Shao17a, XY2011, YZ} etc., here the functions $q_{ij}(x)$ in the $Q$-matrix $Q_{x}=\big(q_{ij}(x)\bigr)$ may be unbounded. Hence the construction of the solution in Theorem \ref{exiuni} is of interest by itself.

In a similar way, one can establish the existence and uniqueness of solution for another widely studied  SFDE with regime-switching. To do so, we introduce some notations. Let $\C$ denote the continuous path space $C([-r,0];\R^d)$ endowed with the uniform topology, i.e., $\|\xi\|_{\infty}=\sup_{-r\leq s\leq 0}|\xi(s)|$ for $\xi\in\C$, where $r\ge 0$ is a constant. Consider the following SFDE:
\begin{equation}\label{1.1delay}
\d X(t)=b(X_t,\La(t),\La(t-r))\d t+\sigma(X_t,\La(t),\La(t-r))\d W(t)
\end{equation} with $X_0=\xi\in \C$, $\La(0)=i\in\S$, and $(\La(t))$ still satisfies \eqref{1.2} as above. Here, $b:\C\times \S\times\S\ra \R^d$, $\sigma:\C\times\S\times\S\ra \R^{d\times d}$, and $X_t\in \C$ is defined by $X_t(\theta)=X(t+\theta)$ for $\theta\in [-r,0]$. Here we regard that $\La(t-r)=i$ for $t-r<0$ when $\La(0)=i$. The following conditions guarantee the existence and uniqueness of the process $(X,\La)$ satisfying \eqref{1.1delay} and \eqref{1.2}.
\begin{itemize}
  \item[$\mathrm{({\bf A1})}$] $b(\cdot,i,j)$ and $\sigma(\cdot,i,j)$ are bounded on bounded subset of $\C$ for every $i,\,j\in\S$. Moreover, there exists a positive constant $K_1$ such that
      \[2\la b(\xi,i,j)-b(\eta,i,j),\xi(0)-\eta(0)\raa +\|\sigma(\xi,i,j)-\sigma(\eta,i,j)\|_{\textrm{HS}}^2\leq K_1\|\xi-\eta\|_{\infty}^2\]
      for all $\xi,\,\eta\in\C$, $i,\,j\in\S$, where $\la\cdot,\cdot\raa$ denotes the Euclidean inner product in $\R^d$, $\|\sigma\|_{\textrm{HS}}^2=\mathrm{trace}(\sigma\sigma^\ast)$ for $\sigma\in \R^{d\times d}$, $\sigma^\ast$ denotes the transpose of $\sigma$.
  \item[$\mathrm{({\bf A2})}$] There exists a positive constant $K_2$ such that $2\la b(\xi,i,j),\xi(0)\raa+\|\sigma(\xi,i,j)\|_{\textrm{HS}}^2\leq K_2(1+\|\xi\|_{\infty}^2)$ for all $\xi\in \C$ and $i$, $j\in \S$.
\end{itemize}

\begin{thm}\label{t2.2delay}
Assume conditions (\textbf{A1}), (\textbf{A2}), (\textbf{Q1}) and (\textbf{Q2'}) hold. Then there exists a unique nonexplosive solution $(X,\La)$ to  regime-switching SFDE \eqref{1.1delay} and \eqref{1.2}.
\end{thm}

Theorem \ref{t2.2delay} can be proved by using the same idea of the argument of Theorem \ref{exiuni}, and hence the proof will be omitted.

In the remainder of this section, we shall present two kinds of preparation results: in the first place, we establish an estimation of exponential functional of a discrete-time Markov chain by the spectrum analysis method; in the second place, we construct two auxiliary Markov chains to control from upper and below the evolution of the state-dependent jumping process $(\La(t))$.

First, let us consider a time-homogeneous Markov chain $(Y_n)_{n\geq 0}$ on the state space $\S=\{1,\ldots,M\}$ with $1<M<\infty$. Denote
\[P_{ij}=\p( Y_1=j|Y_0=i),\qquad i,j\in \S.\]
Assume the transition matrix $P=(P_{ij})$ is a positive matrix, i.e., $P_{ij}>0$, $\forall\,i,j\in \S$.
Let $(\theta(i))_{i\in \S}$ be a series of real numbers. Put
\[\widetilde P_{ij}=e^{\theta(i)}P_{ij},\qquad i,j\in \S,\qquad \widetilde P=(\widetilde P_{ij})_{i,j\in \S}.\]
Denote $\spec(\wt P)$   the spectrum of the linear operator $\wt P$.
Let \[\lambda_1=\max\{\Re(\lambda); \lambda\in \spec(\wt P)\},\] where $\Re( \lambda)$ stands for the real part of the eigenvalue $\lambda$.

\begin{lem}\label{t-2.1}
  Let $\theta:\S\ra \R$. Then there exist two positive constants $K_3$, $K_4$ such that
  \[K_3\lambda_1^n\leq \E_{\mu}\Big[\exp\Big\{\sum_{k=0}^{n-1} \theta(Y_k)\Big\}\Big]\leq K_4\lambda_1^n \]  for every initial probability distribution $\mu$ of $(Y_n)_{n\geq 0}$ when  $n$ large enough. 
\end{lem}

\begin{proof}
  According to the Perron-Frobenius theorem, due to the positivity of $\wt P$, which follows directly from the positivity of $P$, $\lambda_1$ is a simple eigenvalue of $\wt P$, and all the magnitudes of other eigenvalues of $\wt P$ are strictly smaller than $\lambda_1$. Invoking the spectral theory for linear operator in a finite dimensional Banach space (cf. Dunford and Schwartz \cite[Chapter VII, Theorem 8]{DS88}), there exists a family of linear operator $\{E(\lambda);\ \lambda\in \spec (\wt P)\}$ satisfying $E(\lambda)^2=E(\lambda)$, $E(\lambda)E(\tilde \lambda)=0$ if $\lambda\neq \tilde \lambda$, and $I=\sum_{\lambda\in \spec (\wt P)} E(\lambda)$ such that
\begin{equation}
  \label{spec-1} \wt P^n=\lambda_1^n E(\lambda_1)+\!\!\sum_{\lambda\in \spec (\wt P)\backslash \{\lambda_1\}}\!\!\!\sum_{i=0}^{v(\lambda)-1} \frac{(\wt P-\lambda I)^i}{i!}f^{(i)}(\lambda)E(\lambda),
\end{equation}
where $v(\lambda)$ denotes the index of the eigenvalue $\lambda$, which is a constant less than $2M$; the function $f$ is given by $f(x)=x^n$ and $f^{(i)}$ denotes the $i$-th order derivative of $f$. We can rewrite the terms in the summation as
\begin{align*}
  \frac{(\wt P-\lambda I)^i}{i!} f^{(i)}(\lambda)E(\lambda)&=\frac{n!}{i!(n-i)!}\lambda ^{n-i}(\wt P-\lambda I)^i E(\lambda)\\
  &=\lambda_1^n \frac{n!}{i!(n-i)!} \Big(\frac{\lambda}{\lambda_1}\Big)^{n-i}\Big(\frac{\wt P-\lambda I}{\lambda_1}\Big)^i E(\lambda).
\end{align*}
Note that for any $\lambda\in \spec (\wt P)$ with $\lambda\neq \lambda_1$, it holds $|\lambda|<\lambda_1$. Therefore, for any fixed $i<n$,
\[\lim_{n\ra \infty} \frac{n!}{i! (n-i)!}\Big(\frac{|\lambda|}{\lambda_1}\Big)^{n-i}=0.\]
Consequently, by \eqref{spec-1}, for any initial probability measure $\mu$ of $(Y_n)$ on $\S$,
\[\mu \wt P^n \mathbf 1=\lambda_1^n\mu E(\lambda_1)\mathbf 1+\sum_{\lambda\in \spec(\wt P)\backslash\{\lambda_1\}}\sum_{i=0}^{v(\lambda)-1}\mu\Big[\frac{(\wt P-\lambda I)^i}{i!}f^{(i)}(\lambda)E(\lambda)\mathbf 1\Big].\]
Hence, there exist two positive constants $K_3$, $K_4$ independent of the initial distribution $\mu$  such that
\[K_3\lambda_1^n\leq \mu \wt P^n\mathbf 1\leq K_4 \lambda_1^n,\qquad \text{for $n$ large enough}.\]
Invoking the definition of $\wt P$, this can be written in the expectation form as
\begin{equation}
  \label{2.2}
  K_3 \lambda_1^n\leq \E_{\mu}\Big[\exp\Big\{\sum_{k=0}^{n-1} \theta(Y_k)\Big\}\Big]\leq K_4 \lambda_1^n, \qquad \text{for $n$ large enough},
\end{equation}
which is just desired conclusion.
\end{proof}

Next, employing the idea of Shao \cite{Shao17}, we go to construct two auxiliary continuous-time Markov chains $(\bar\La(t))$ and $(\La^*(t))$
such that $\La^*(t)\leq\La(t)\leq \bar\La(t)$, $t\geq 0$, a.s., under some
appropriate conditions. Our stochastic comparisons are based on
Skorokhod's representation of $(\La(t))$ in
terms of the Poisson random measure by following the line of
\cite[Chapter II-2.1]{Sk89} or \cite{YZ}. To focus on the idea, we first consider the special situation that $\S$ consists of only two points, i.e., $\S=\{1,2\}$. To do so, we further
assume the following condition holds:
\begin{itemize}
   \item[({\bf Q3})] For each $x\in\R^d$, the $Q$-matrix $Q_{x}=\big(q_{ij}(x)\bigr)$ is irreducible.
\end{itemize}

According to the conservativeness of $Q_x$,
one has $q_1(x):=-q_{11}(x)=q_{12}(x)$, $x\in \R^d$.
For each $x\in \R^d$, let
\begin{align*}
  \Gamma_{12}(x):=[0,q_{12}(x))~~~~~\mbox{ and }~~~~~
  \Gamma_{21}(x):=[q_1(x),q_1(x)+q_{21}(x)).
\end{align*}
Obviously, the length of  $\Gamma_{12}(x)$ and  $\Gamma_{21}(x)$ is
$q_{12}(x)$ and $q_{21}(x)$, respectively.
 Define a function $\R^d\times
\S\times \R\ni(x,i,u)\mapsto h(x,i,u)\in\R$ by
\begin{equation*}
h(x,i,u)=(-1)^{1+i}\Big\{{\bf1}_{\{i=1\}}{\bf1}_{\Gamma_{ii+1}(x)}(u)+{\bf1}_{\{i=2\}}{\bf1}_{\Gamma_{ii-1}(x)}(u)\Big\}.
\end{equation*}
Then, as in the proof of Theorem \ref{exiuni}, $(\La_t)$ solves the following stochastic differential
equation (SDE for short)
\begin{equation}\label{2.3}
\d \La(t)=\int_{[0,L]} h(X(t),\La(t-),u)N(\d t,\d u),~~
~t>0,~~~\La(0)=i_0\in\S.
\end{equation}
Herein, $L:=2H$ with $H$ being introduced in condition ({\bf Q2}) and $N(\d
t,\d u)$ stands for a Poisson random measure with intensity $\d
t\times \mathbf{m}(\d u)$, in which $\mathbf{m}(\d u)$ signifies the
Lebesgue measure on $[0,L]$. Let $p(t)$ be the stationary Poisson
point process corresponding to the Poisson random measure $N(\d t,\d
u)$ so that $N([0,t)\times A)=\sum_{s\leq t}{\bf 1}_A(p(s))$ for
$A\in\mathscr{B}(\R)$.

Due to the finiteness of $\mathbf{m}(\d u)$ on $[0, L]$, there is
only finite number of jumps of the process $(p(t))$ in each finite
time interval. Let $ 0=\zeta_0<\zeta_1<\cdots<\zeta_n<\cdots$ be the
enumeration of all jumps of $(p(t))$. It holds that $\lim_{n\ra
\infty}\zeta_n=+\infty$ a.s. From \eqref{2.3}, it follows that
\begin{equation*}
\La(t)=i_0+\sum_{s\leq
t}h(X(s),\La(s-),p(s)){\bf1}_{[0,L]}(p(s)),~~
~t>0,~~~i_0\in\S,
\end{equation*}
which implies  that $(\La(t))$ may have a jump at only $\zeta_i$ (i.e.
$\La(\zeta_i)\neq \La(\zeta_i-)$) provided that
$p(\zeta_i)\in[0,L]$. So the collection of all jumping
times of $(\La(t))$ is a subset of $\{\zeta_1,\zeta_2,\cdots\}$.
Subsequently, this basic fact will be used frequently without mentioning it again.

Let
\begin{equation}\label{q-up}
\bar q_{12}:=\sup_{x\in\R^d} q_{12}(x), \ \bar
q_{21}:=\inf_{x\in\R^d}q_{21}(x), \ \bar q_1:=-\bar q_{11}:=\bar
q_{12}, \ \bar q_2:=-\bar q_{22}:=\bar q_{21},
\end{equation}
and
\begin{equation}\label{q-low}
 q_{12}^*:=\inf_{x\in\R^d} q_{12}(x), \
q_{21}^*:=\sup_{x\in\R^d}q_{21}(x), \ q_1^*:=-q_{11}^*:=  q_{12}^*, \
 q_2^*:=- q_{22}^*:=q_{21}^*.
\end{equation}
Let
\begin{equation*}
\bar \Gamma_{12}:=[0,\bar q_{12}), \ \bar \Gamma_{21}:=[\bar
q_{12}, \ \bar q_{12}+\bar q_{21}), \ \Gamma_{12}^*:=[0, q_{12}^*), \
\Gamma_{21}^*:=[q_{12}^*, q_{12}^*+q_{21}^*).
\end{equation*}
Using the same Poisson random measure $N(\d t,\d u)$  given in
\eqref{2.3}, we define two auxiliary Markov chains $(\bar \La(t))$ and $(\La^*(t))$ by the following SDEs:
\begin{equation}\label{2.4}
\d \bar\La(t)=\int_{[0, L]} \bar g(\bar\La({t-}),u)N(\d t,\d
u),~~~~~t>0,~~~~~\quad \bar\La(0)=\La(0),
\end{equation}
and
\begin{equation}\label{2.5}
\d  \La^*(t)=\int_{[0,  L]}   g^*(\La^*(t-),u)N(\d t, \d
u),~~~~~t>0,~~~~~\quad  \La^*(0)=\La(0),
\end{equation}
where, for $i\in\S$,
\begin{equation*}
\bar
g(i,u)=(-1)^{1+i}
\Big\{{\bf1}_{\{i=1\}}{\bf1}_{\bar\Gamma_{ii+1}}(u)
+{\bf1}_{\{i=2\}}{\bf1}_{\bar\Gamma_{ii-1}}(u)\Big\},~~~u\in[0,
L],
\end{equation*}
and
\begin{equation*}
g^*(i,u)=(-1)^{1+i}\Big\{{\bf1}_{\{i=1\}}{\bf1}_{\Gamma_{ii+1}^*}(u)
+{\bf1}_{\{i=2\}}{\bf1}_{\Gamma_{ii-1}^*}(u)\Big\},~~~u\in[0,
L].
\end{equation*}
Then, according to Skorokhod's representation, $(\bar\La_t)$ and
$(\La_t^*)$ are continuous-time Markov chains on $\S=\{1,2\}$ generated by the $Q$-matrices $\bar Q=(\bar
q_{ij})_{1\leq,i,j\leq2}$ and $Q^*=(q_{ij}^*)_{1\leq,i,j\leq2}$,
respectively.

\begin{lem}\label{t-2.2}
\begin{itemize}
  \item[$(i)$] If $\bar q_{21}>0$  and
\begin{equation}\label{2.6}
  \bar q_{12}+\bar q_{21}\leq q_{12}(x)+q_{21}(x),\qquad x\in \R^d,
\end{equation}
then
  $\La(t)\leq \bar\La(t)$ for all $t\geq 0$ a.s.
  \item[$(ii)$] If $ q_{12}^\ast>0$  and
\begin{equation}\label{2.7}
  q_{12}^*+ q_{21}^* \geq q_{12}(x)+q_{21}(x), \qquad x\in \R^d,
\end{equation}
then $\La^*(t)\leq \La(t)$ for all $t\geq 0$ a.s.
\end{itemize}
\end{lem}

\begin{proof}
  Here we include a proof for the completeness and the ease of the readers. We shall only prove assertion (i) as assertion (ii) can be proved in a similar way. Since there is no jump during the open interval $(\zeta_k,\zeta_{k+1})$, we only need to prove (i) at $\zeta_k$, $k\geq 1$. To this aim, we consider separately three different cases.

\noindent{\bf Case $1$: $\La(\zeta_k)=\bar\La(\zeta_k)=1$, $k\geq 1$.}
In this case, we deduce from \eqref{2.3} and \eqref{2.4} that
  \begin{equation*}
    \La(\zeta_{k+1})=1+\mathbf
    1_{\Gamma_{12}(X({\zeta_{k+1}}))}(p(\zeta_{k+1}))~~~~\mbox{ and } ~~~~
    \bar\La(\zeta_{k+1})=1+\mathbf 1_{\bar\Gamma_{12}}(p(\zeta_{k+1})).
  \end{equation*}
According to the notion of $\bar q_{12}$, one clearly has
$q_{12}(x)\leq \bar q_{12}$, $x\in\R^d$, which implies that
$\Gamma_{12}(X({\zeta_{k+1}}))\subset\bar \Gamma_{12}$, a.s. Whence,
$\La(\zeta_{k+1})\leq \bar\La(\zeta_{k+1})$, a.s.

 \noindent{\bf Case $2$: $\La(\zeta_k)=\bar\La(\zeta_k)=2$, $k\geq 1$.}   Concerning such case, we also obtain from \eqref{2.3} and \eqref{2.4} that
  \begin{equation}\label{w0}
    \La(\zeta_{k+1})=2-\mathbf 1_{\Gamma_{21}(X({\zeta_{k+1}}))}(p(\zeta_{k+1}))~~~~\mbox{ and } ~~~~
    \bar\La(\zeta_{k+1})=2-\mathbf 1_{\bar\Gamma_{21}}(p(\zeta_{k+1})).
  \end{equation}
If $p_1(\zeta_{k+1})\notin\bar\Gamma_{21}$, then, from \eqref{w0},
one has $\bar\La(\zeta_{k+1})=2$ so that $\La(\zeta_{k+1})\leq
\bar\La(\zeta_{k+1})$ due to the fact that $\La(\zeta_{k+1})\leq 2$.
Next, we proceed to deal with the case
$p(\zeta_{k+1})\in \bar \Gamma_{21}$, which of course leads to
$\bar\La(\zeta_{k+1})=1$ in view of \eqref{w0}, and $\bar q_{12}\leq
p(\zeta_{k+1})<\bar q_{12}+\bar q_{21}$. Employing the assumption
\eqref{2.6} and utilizing the fact that $q_{12}(X({\zeta_{k+1}}))\leq
\bar q_{12}$, we arrive at $q_{12}(X({\zeta_{k+1}}))\leq p(\zeta_{k+1})<
q_{12}(X({\zeta_{k+1}}))+ q_{21}(X({\zeta_{k+1}}))$, namely,
$p(\zeta_{k+1})\in \Gamma_{12}(X({\zeta_{k+1}}))$. As a consequence,
$\La(\zeta_{k+1})=\bar\La(\zeta_{k+1})=1$.

\noindent{\bf Case $3$: $\La(\zeta_k)=1$, $\bar\La(\zeta_k)=2$,
 $k\geq 1$.}
 For this setup, it follows from \eqref{2.3} and \eqref{2.4} that
  \begin{equation}\label{f1}
    \La(\zeta_{k+1})=1+\mathbf
    1_{\Gamma_{12}(X({\zeta_{k+1}}))}(p(\zeta_{k+1}))~~~~\mbox{ and
    }~~~~
    \bar \La(\zeta_{k+1})=2-
        \mathbf 1_{\bar\Gamma_{21}}(p(\zeta_{k+1})).
  \end{equation}
From \eqref{f1}, it is easy to see that $\La(\zeta_{k+1})\leq\bar
\La(\zeta_{k+1})$ if $p(\zeta_{k+1})\notin \bar \Gamma_{21}$. Now,
if $p(\zeta_{k+1})\in \bar\Gamma_{21}$, then we infer that $\bar
\La(\zeta_{k+1})=1$ and that $\bar q_{12}\leq p(\zeta_{k+1})< \bar
q_{12}+\bar q_{21}$. Hence, one has
$p(\zeta_{k+1})\ge q_{12}(X({\zeta_{k+1}}))$; in other words,
$p(\zeta_{k+1})\notin\Gamma_{12}(X({\zeta_{k+1}}))$. As a result, we
obtain from \eqref{f1} that $\La(\zeta_{k+1})=\bar
\La(\zeta_{k+1})=1$.

The desired result follows immediately by summing up the above three cases.
\end{proof}

{ We now proceed to the general situation that $\S$ could own more than two states}, i.e., $\S=\{1,\ldots,M\}$. To this end, we employ the coupling method and especially construct the so-called order-preserving couplings. To do so, we need some preparation.

Let $\Lambda^{(1)}$ and $\Lambda^{(2)}$ be two continuous-time
Markov chains defined by two generators
$Q^{(1)}=\bigl(q_{ij}^{(1)}\bigr)$ and
$Q^{(2)}=\bigl(q_{ij}^{(2)}\bigr)$ on the state space $\S$,
respectively. Note that $Q^{(1)}$ and $Q^{(2)}$ are called
$Q$-matrices in \cite{C2004, Z1996, Z1998}. A continuous-time Markov
chain $(\Lambda^{(1)},\Lambda^{(2)})$ on the product space $\S
\times \S$ is called a coupling of $\Lambda^{(1)}$ and
$\Lambda^{(2)}$, if the following marginality holds for any $t\ge
0$, $i_1, i_2 \in \S$ and $B_1, B_2 \subset \S$,
\beq{(OPC0)}
\barray\ad {\p}^{(i_{1},i_{2})}\bigl((\Lambda^{(1)}(t),
\Lambda^{(2)}(t))\in B_1\times
\S\bigr)=\p^{(i_{1})}(\Lambda^{(1)}(t)\in B_1),\\
\ad
{\p}^{(i_{1},i_{2})}\bigl((\Lambda^{(1)}(t),\Lambda^{(2)}(t))\in
\S \times B_2\bigr)=\p^{(i_{2})}(\Lambda^{(2)}(t)\in B_2),
\earray\eeq where the superscript in ${\p}^{(i_{1},i_{2})}$ and $\p^{(i_1)}$ is used to emphasize the initial value of the corresponding process.  From \cite{C2004} we know that
constructing a coupling Markov chain $(\Lambda^{(1)},\Lambda^{(2)})$
is equivalent to constructing a coupling generator $\wdt{Q}$ on the
finite state space $\S \times \S$, and such a generator $\wdt{Q}$
is called a coupling of $Q^{(1)}$ and $Q^{(2)}$. For given two
generators (or two $Q$-matrices), one can construct many their
coupling generators (or coupling $Q$-matrices); see \cite{C2004} for
more examples and explanation. In what follows we are especially
interested in the order-preserving couplings. On the product space
$\S \times \S$, an order-preserving coupling $\wdt{Q}$ of
$Q^{(1)}$ and $Q^{(2)}$ means that the corresponding Markov chain generated by $\wdt{Q}$ satisfies
\beq{(OPC1)}
{\p}^{(i_{1},i_{2})}\bigl(\Lambda^{(1)}(t)\le
\Lambda^{(2)}(t), \ \forall\, t\geq 0\bigr)=1,   \quad i_{1}\le i_{2} \in
\S.\eeq  See \cite[Chapter 5]{C2004} for the details about the
coupling $Q$-matrices and related materials. For more general case,
the construction of order-preserving couplings was studied in
\cite{Z1996, Z1998}. In particular, we have the following lemma from
the aforementioned three references.

\begin{lem}\label{OPC1}
If the generators $Q^{(1)}$ and $Q^{(2)}$ on $\S$
satisfy that \beq{(OPC2)}\barray \ad\sum_{l\ge m}q_{i_{1}l}^{(1)}
\le \sum_{l\ge m}q_{i_{2}l}^{(2)} \ \ \hbox{for all} \ \ i_{1}\le
i_{2} <m \ \
\hbox{and}\\
\ad\sum_{l\le m}q_{i_{1}l}^{(1)} \ge \sum_{l\le m}q_{i_{2}l}^{(2)} \
\ \hbox{for all} \ \ m< i_{1}\le i_{2},\earray\eeq there exists an
order-preserving coupling $Q$-matrix $\wdt{Q}$ on $\S \times \S$
and hence (\ref{(OPC1)}) holds. \end{lem}

\begin{Assumption}
\label{OPC2}{\rm
Assume that there exists a generator $\bar{Q}=\bigl(\bar{q}_{i,j}\bigr)$ on $\S$ such that
the following bounds hold:
\beq{(OPC3)}\barray \ad \sup_{x\in \R^d}\sum_{l\ge m}q_{i_{1}l}(x)
\le \sum_{l\ge m}\bar{q}_{i_{2}l} \ \ \hbox{for all} \ \ i_{1}\le
i_{2} <m \ \
\hbox{and}\\
\ad \inf_{x\in \R^d}\sum_{l\le m}q_{i_{1}l}(x) \ge \sum_{l\le m}\bar{q}_{i_{2}l} \ \ \hbox{for all} \ \ m< i_{1}\le
i_{2},\earray\eeq where the matrix $\bigl(q_{ij}(x)\bigr)$ is given in (\ref{1.2}).}\end{Assumption}

If two generators $Q^{(1)}$ and $Q^{(2)}$ satisfy (\ref{(OPC2)}), we simply
write $Q^{(1)}\preceq Q^{(2)}$. For convenience, with a slight abuse
of notation, we denote the matrix $\bigl(q_{ij}(x)\bigr)$ by $Q_x$.
So Assumption~\ref{OPC2} means that for each $x\in \R^{d}$,
$Q_x\preceq \bar{Q}$. By Lemma~\ref{OPC1}, for each $x\in
\R^{d}$, there exists an order-preserving coupling of $Q_x$ and $\bar{Q}$ given in Assumption~\ref{OPC2}. Namely, for each $x\in \R^{d}$, there exists a $Q$-matrix on
$\S\times \S$ such that this $Q$-matrix is an order-preserving
coupling of $Q_x$ and $\bar{Q}$. In fact, such an order-preserving
coupling was constructed explicitly in \cite{Z1996, Z1998}; see also \cite[p.
221]{C2004}. For definiteness, we choose one such coupling and denote it by
$\wdt{Q}(x)=\bigl(\wdt{q}(i,j;m,n)(x)\bigr)$, which
can be expressed explicitly. For the sake of completeness and also
certain subsequent application, we sketch the construction of the coupling
$\wdt{Q}(x)$ here though a method which is essentially not new
(cf. \cite{Z1996, Z1998}).

As mentioned in \cite{Z1996}, we can define the basic coupling of $Q_x$ and $\bar{Q}$ for the points $(i,j)
\in \S\times \S$ with $i>j$ as follows:
\beq{basic}\left
\{\begin{array}{ll} \disp\wdt{q}(i,j;m,n)(x)=
\bigl(q_{ik}(x)-\bar{q}_{jk}\bigr)^{+}, & m=k, n=j, k\neq i,\\
\disp\wdt{q}(i,j;m,n)(x)=\bigl(\bar{q}_{jk}-q_{ik}(x)\bigr)^{+}, & m=i, n=k, k\neq j,\\
\disp\wdt{q}(i,j;m,n)(x)=q_{ik}(x)\wedge \bar{q}_{jk}, & m=k, n=k, (k,k)\neq (i,j),\\
\disp\wdt{q}(i,j;m,n)(x)=0, & \hbox{other}\ \ (m,n) \neq
(i,j),
\end{array}
\right.\eeq and $
\disp\wdt{q}(i,j;i,j)(x):=-\sum_{(m,n) \neq (i,j)}
\wdt{q}(i,j;m,n)(x)$.

Next, we construct the order-preserving
coupling for the points $(i,j) \in \S\times \S$ with $i\le j$, which is the key point to construct a coupling $(\La,\bar \La)$ so that $\p(\La(t)\leq \bar\La(t),\ \forall\, t\geq 0)=1$. For each $n\in \S$, define \beq{ZCa}\barray\ad
a_{nn}(x)=\left \{\begin{array}{ll}q_{in}(x), & n\neq i,\\
0, & n=i,\end{array} \right.\\
\ad
b_{nn}(x)=\left \{\begin{array}{ll}\bar{q}_{jn}, & n\neq j,\\
0, & n=j.\end{array} \right.
\earray\eeq
Then, define the
sequences $\{a_{mn}(x)\}$, $\{b_{mn}(x)\}$ $(m\le n, n\in \S)$: for
$m=n-1$, $n-2$, $\cdots$, $1$ successively as
\beq{ZCb}\barray\ad
a_{mn}(x)=\bigl(a_{m,n-1}(x)\bigr)^{+}-\bigl(b_{m,n-1}(x)\bigr)^{+},
\\
\ad b_{mn}(x)=\bigl(b_{m+1,n}(x)\bigr)^{+}-\bigl(a_{m+1,n}(x)\bigr)^{+}.\earray\eeq
Here and hereafter, $a^{+}=\max\{a,0\}$. Let us give some explanation on this definition
procedure. Clearly, we can define
$a_{12}(x)$ and $b_{12}(x)$ with (\ref{ZCb}) by the well-defined
$a_{11}(x)$, $b_{11}(x)$, $a_{22}(x)$ and $b_{22}(x)$. Suppose the
 $\{a_{m^{\prime}n^{\prime}}(x)\}$,
$\{b_{m^{\prime}n^{\prime}}(x)\}$ $(m^{\prime}\le n^{\prime})$ have
been defined successively for $n^{\prime}=1$, $2$, $\ldots$,
$n-1$. With (\ref{ZCb}) we can further define the case of
$n^{\prime}=n$.
Although $b_{nn}(x)$ is independent of $x$, $b_{mn}(x)$ is $x$-dependent in general. Finally, with the well-defined sequences
$\{a_{mn}(x)\}$, $\{b_{mn}(x)\}$ $(m\le n, n\in \S)$, the desired
coupling is given by \beq{ZCc}\left
\{\begin{array}{ll} \disp\wdt{q}(i,j;m,n)(x)=
\bigl(a_{mn}(x)\bigr)^{+}\wedge \bigl(b_{mn}(x)\bigr)^{+}, & m\le n, m\neq i, n\neq j,\\
\disp\wdt{q}(i,j;\,\,i,\,n)(x)=\bar{q}_{jn}-\sum_{1\le m \le n,\,
m\neq i}\!\!\!\wdt{q}(i,j;m,n)(x), & i\le n \neq j,\\
\disp\wdt{q}(i,j;m,j)(x)=q_{im}(x)-\sum_{n\ge m,\, n\neq
j}\!\!\!\wdt{q}(i,j;m,n)(x), & i\neq m \le j,\\
\disp\wdt{q}(i,j;m,n)(x)=0, & \hbox{other}\ \ (m,n) \neq
(i,j),
\end{array}
\right.\eeq and \[
\disp\wdt{q}(i,j;i,j)(x)=-\sum_{(m,n) \neq (i,j),\, m\le n}
\wdt{q}(i,j;m,n)(x).\]

For every fixed $x\in\R^d$, let $(\La,\bar \La)$ be the continuous-time Markov chain determined by the coupling operator $\wdt Q$ defined in \eqref{ZCc} with $\La(0)\leq \bar \La(0)$. As shown in \cite{Z1996}, the construction of $\widetilde q(i,j;m,n)$ guarantees that the process $\La $ can never jump to the front of $\bar \La $ a.s., i.e. $\p(\La(t)\leq \bar \La(t),\ \forall\,t\geq 0)=1$.

Consequently, the order-preserving coupling $\bigl(X,\Lambda,\bar{\Lambda}\bigr)$ is constructed as follows.
Let $X$ satisfy SDE
(\ref{1.1}) and $(\La,\bar \La)$ be a jumping process on $\S\times\S$ with $(\La(0),\bar \La(0))=(i_0,j_0)$ satisfying
\beq{(OPC4)}
\begin{array}{l}
\p\{(\Lambda(t+\Delta),\bar{\Lambda}(t+\Delta))=(m,n)|
(\Lambda(t),\bar{\Lambda}(t))=(i,j),X(t)=x\}\\
\quad =\left\{\begin{array}{ll}\wdt{q}(i,j;m,n)(x)\Delta
+o(\Delta),
& \hbox{if}\, (m,n)\neq (i,j), \\
1+\wdt{q}(i,j;m,n)(x)\Delta +o(\Delta), & \hbox{if}\,
(m,n)=(i,j),
\end{array} \right.
\end{array}
\eeq provided $\Delta \downarrow 0$. Here $\wdt q(i,j;m,n)$ is determined by \eqref{basic} or \eqref{ZCc} according to $i_0>j_0$ or not.

\begin{lem}\label{OPC3}
Suppose that (\textbf{H1}), (\textbf{H2}), (\textbf{Q1}),  (\textbf{Q2'}) and  Assumption~\ref{OPC2} holds. Then the coupling process
$\bigl(X,\Lambda,\bar{\Lambda}\bigr)$ satisfying \eqref{1.1} and \eqref{(OPC4)} exists, and further
\beq{(OPC7)}  {\p}^{(x,i,j)}\bigl(\Lambda(t)\le
\bar{\Lambda}(t),\ \forall\,t\geq 0\bigr)=1,  \quad x\in \R^{d},\quad
i\le j \in \S. \eeq
In addition, suppose that $\bar{Q}$ is irreducible, then the invariant measure $\bar{\mu}=(\bar{\mu}_{1},\ldots,\bar{\mu}_{M})$ of $\bar{\Lambda}$  exists, and
for each increasing function $h$ on $\S$ and each $(x,i)\in
\R^{d}\times \S$, \beq{(OPC8)}{\p}^{(x,i)}\biggl(\limsup_{t\to
\infty}\frac{1}{t}\int_{0}^{t}h(\Lambda(s))\d s \le \sum_{m\in
\S}h(m)\bar{\mu}_{m}\biggr)=1,\eeq for every initial value $\bigl(X(0),\Lambda(0) \bigr)=(x,i)$.
\end{lem}

\begin{proof} By the definition of $a_{nn}(x)$, $a_{mn}(x)$ and $\wdt q(i,j;m,n)$, it is easy to check the validation of  conditions (\textbf{Q1}) and (\textbf{Q2'})  for $\wdt Q(x)$. Therefore, Theorem \ref{exiuni} ensures that the system $(X,\La,\bar \La)$ satisfying \eqref{1.1} and \eqref{(OPC4)} exists.
Although the transition rate matrix of $(\La,\bar \La)$ now depends on $X$ which is time varying, the construction of $\wdt Q_x$ still can ensure that $\La(t)$ cannot jump to the front of $\bar \La(t)$ a.s. if $\La(0)\leq \bar \La(0)$.
Hence, (\ref{(OPC7)}) holds.

Using the right continuity of
$\bigl(X,\Lambda,\bar{\Lambda}\bigr)$, from (\ref{(OPC7)}) we
obtain that for each given increasing function $h$ on $\S$, $x\in
\R^{d}$ and $i\le j \in \S$, \beq{(OPC9)}
{\p}^{(x,i,j)}\biggl(\frac{1}{t}\int_{0}^{t}h(\Lambda(s))\d s\le
\frac{1}{t}\int_{0}^{t}h(\bar{\Lambda}(s))\d s, \ \ \forall\,
t\ge 0\biggr)=1.\eeq Since $\bar{Q}$ is irreducible and $\S$ is a finite state space, the
associated Markov chain $\bar{\Lambda}$ is ergodic with the
invariant probability measure given by
$\bar{\mu}=(\bar{\mu}_{1},\ldots,\bar{\mu}_{M})$. By the
ergodic property of the Markov chain, we have
\beq{(OPC10)} {\p} \biggl(\lim_{t\to
\infty}\frac{1}{t}\int_{0}^{t}h(\bar{\Lambda}(s))\d s = \sum_{m\in
\S}h(m)\bar{\mu}_{m}\biggr)=1.
\eeq
For any arbitrarily given $\delta>0$, by Egorov's theorem, we then get
\beq{(OPC11)}
{\p} \biggl(\frac{1}{t}\int_{0}^{t}h(\bar{\Lambda}(s))\d s \to
\sum_{m\in \S}h(m)\bar{\mu}_{m} \ \ \hbox{uniformly as} \ \ t \to
\infty \biggr)\ge 1-\delta. \eeq Thus, for any given
$\varepsilon>0$, there exists a $T(\varepsilon)>0$ such that
\beq{(OPC12)}
1-\delta   \le
{\p} \biggl(\frac{1}{t}\int_{0}^{t}h(\bar{\Lambda}(s))\d s \le
\sum_{m\in \S}h(m)\bar{\mu}_{m}+\varepsilon \ \ \hbox{for all} \ \
t\ge T(\varepsilon) \biggr).\eeq
Therefore,
combining (\ref{(OPC9)}) and (\ref{(OPC12)}), we derive that for
every increasing function $h$ on $\S$, $x\in \R^{d}$ and $i\le
j \in \S$, \beq{(OPC13)}
 {\p}^{(x,i,j)}\biggl(\frac{1}{t}\int_{0}^{t}h(\Lambda(s))\d s\le
\sum_{m\in \S}h(m)\bar{\mu}_{m}+\varepsilon \ \ \hbox{for all} \ \
t \ge T(\varepsilon)\biggr)\ge 1-\delta,
\eeq
which yields that
\beq{(OPC14)}
{\p}^{(x,i)}\biggl(\frac{1}{t}\int_{0}^{t}h(\Lambda(s))\d s\le \sum_{m\in
\S}h(m)\bar{\mu}_{m}+\varepsilon \ \ \hbox{for all} \ \ t \ge
T(\varepsilon)\biggr)\ge 1-\delta \eeq since the left hand side of \eqref{(OPC13)} does not depend on $j$. So,
\beq{(OPC15)}
{\p}^{(x,i)}\biggl(\limsup_{t\to\infty}\frac{1}{t}
\int_{0}^{t}h(\Lambda(s))\d s\le \sum_{m\in
\S}h(m)\bar{\mu}_{m}+\varepsilon\biggr)\ge 1-\delta.\eeq Finally,
letting $\varepsilon$ and $\delta$ tend to $0$, we arrive at (\ref{(OPC8)}).
\end{proof}

Such a Markov chain $\bar{\Lambda}$ having the properties stated in
Lemma~\ref{OPC3} will be called a upper control Markov chain. It
will serve as an {\it upper envelop} as mentioned above.
On the other hand, a lower control Markov chain $\Lambda^{\ast}$ can be constructed under proper conditions.

\begin{Assumption}\label{OPC4}{\rm
Assume that there exists a generator
$Q^{\ast}=\bigl(q^{\ast}_{i,j}\bigr)$ on $\S$ such that the following bounds hold: \beq{(OPC17)}\barray \ad \sum_{j\ge
m}q^{\ast}_{i_{1},j} \le \inf_{x\in \R^d}\sum_{j\ge m}q_{i_{2}j}(x)
\ \ \hbox{for all} \ \ i_{1}\le i_{2} <m \ \
\hbox{and}\\
\ad \sum_{j\le m}q^{\ast}_{i_{1},j} \ge \sup_{x\in \R^d}\sum_{j\le
m}q_{i_{2}j}(x) \ \ \hbox{for all} \ \ m< i_{1}\le
i_{2}.\earray\eeq}\end{Assumption}
For each $x\in \R^{d}$, an order-preserving coupling of $Q^{\ast}$ and $Q(x)$ can be constructed explicitly
(see \cite[p. 221]{C2004}). Actually, replace the $(q_{ij}(x))$ and
$(\bar{q}_{i,j})$ in (\ref{ZCa})--(\ref{ZCc}) by $(q^{\ast}_{i,j})$
and $(q_{ij}(x))$ respectively, we can construct a coupling operator $\wdt{Q}(x)=(\wdt q(i,j;m,n)(x))$
for each $x\in \R^{d}$ and $i\geq j$.  When $i<j$, we still use the basic coupling given in \eqref{basic} by replacing $\bar q_{j,k}$ with $q^\ast_{j,k}$.  We now
construct an order-preserving coupling  process $\bigl(X,\Lambda^{\ast},\Lambda\bigr)$
as follows. Let   $X$ satisfy  SDE (\ref{1.1}) and $(\La^\ast, \La)$ be a jumping process on $\S\times\S$ with $(\La^\ast(0),\La(0))=(i_0,j_0)$ satisfying
\beq{(OPC18)}
\begin{array}{l}
\p\{(\Lambda^{\ast}(t+\Delta),\Lambda(t+\Delta))=(m,n)|
(\Lambda^{\ast}(t),\Lambda(t))=(i,j),X(t)=x\}\\
\quad =\left \{\begin{array}{ll}\wdh{q}(i,j;m,n)(x)\Delta
+o(\Delta),
& \hbox{if}\, (m,n)\neq (i,j), \\
1+\wdh{q}(i,j;m,n)(x)\Delta +o(\Delta), & \hbox{if}\,
(m,n)=(i,j),
\end{array} \right.
\end{array}
\eeq provided $\Delta \downarrow 0$. Similar to Lemma \ref{OPC3}, we can prove the
following lemma.

\begin{lem}\label{OPC5}
Suppose that (\textbf{H1}), (\textbf{H2}), (\textbf{Q1}),  (\textbf{Q2'}) and  Assumption~\ref{OPC4} holds. Then the coupling process
$\bigl(X,\Lambda^\ast, \Lambda\bigr)$ satisfying \eqref{1.1} and \eqref{(OPC4)} exists, and further
\beq{(OPC19)}  {\p}^{(x,j,i)}\bigl(
\Lambda(t)\geq \Lambda^{\ast}(t),\ \forall\,t\geq 0\bigr)=1 \   \text{with $(X(0), \La^\ast(0),\La(0))=(x,j,i)$, $i\geq j$}.
\eeq  In addition, suppose that $Q^{\ast}$ is irreducible, then for each
increasing function $h$ on $\S$ and each $(x,i)\in \R^{d}\times
\S$, \beq{(OPC20)}{\p}^{(x,i)}\biggl(\liminf_{t\to
\infty}\frac{1}{t}\int_{0}^{t}h(\Lambda(s))\d s \ge \sum_{m\in
\S}h(m)\mu^{\ast}_{m}\biggr)=1,\eeq where $\mu^{\ast}=(\mu^{\ast}_{1},\ldots,\mu^{\ast}_{M})$ is the invariant
probability measure associated with $\Lambda^{\ast}$.
\end{lem}

In what follows, we provide two concrete examples to illustrate the application of order-preserving couplings to construct upper control and lower control Markov chains for the jump component of state-dependent regime-switching processes.

\begin{Example}\label{OPC6}{\rm
Consider the case $d=1$, $\S=\{1,2\}$.  Let   $Q_x$ in (\ref{1.2})
be given by
$$Q_x=\bigl(q_{ij}(x)\bigr)=\left
(\begin{array}{cc}
  \sin^{2}x-2& 2-\sin^{2}x \\
1+|\cos x|& -1-|\cos x|
\end{array} \right ).$$ Meanwhile, we choose
$$\bar{Q}=(\bar{q}_{i,j})=\left (\begin{array}{cc}
  -2& 2\\
1& -1
\end{array} \right) \qquad \hbox{and}
\qquad Q^{\ast}=(q^{\ast}_{i,j})=
\left (\begin{array}{cc}-1& 1\\
2& -2\end{array} \right).$$ It can be verified that $Q_x \preceq
\bar{Q}$ and $Q^{\ast} \preceq Q(x)$ for all $x \in \R$ and
that both $\bar{Q}$ and $Q^{\ast}$ are irreducible and their
associated invariant probability measures are given by
$\bar{\mu}=(\bar{\mu}_{1},\bar{\mu}_{2})=({1}/{3},{2}/{3})$ and
$\mu^\ast=(\mu^{\ast}_{1},\mu^{\ast}_{2})=({2}/{3},{1}/{3})$
respectively. Thus, for the system $(X,\Lambda)$ satisfying
(\ref{1.1}) and (\ref{1.2}) and any given increasing function
$h$ on $\S=\{1,2\}$, by virtue of Lemma~\ref{OPC3}, we have
\beq{(OPC8a)}{\p}^{(x,i)}\biggl(\limsup_{t\to
\infty}\frac{1}{t}\int_{0}^{t}h(\Lambda(s))\d s \le
\frac{h(1)+2h(2)}{3}\biggr)=1;\eeq and by virtue of
Lemma~\ref{OPC5}, we have
\beq{(OPC8a2)}{\p}^{(x,i)}\biggl(\liminf_{t\to
\infty}\frac{1}{t}\int_{0}^{t}h(\Lambda(s))\d s \ge
\frac{2h(1)+h(2)}{3}\biggr)=1.\eeq}\end{Example}

\begin{Example}\label{OPC7}{\rm
Consider the case $d=1$, $\S=\{1,2,3\}$. Let $Q(x)$ in (\ref{1.2})
be defined by
$$Q(x)=\bigl(q_{ij}(x)\bigr)=
\left(\begin{array}{ccc} -3-|\cos x|+\sin^{2} x & 1+|\cos x|
  & 2-\sin^{2} x\\
1+\frac{x^2}{1+x^2} & -2-\frac{x^2}{1+x^2} & 1\\
2+|\sin x| & 1+\frac{|x|}{1+|x|} & -3-|\sin x|-\frac{|x|}{1+|x|}
\end{array} \right).$$ Meanwhile, we choose
$$\bar{Q}=(\bar{q}_{i,j})=
\left(\begin{array}{ccc}
  -4 & 2 & 2\\
1 & -3 & 2\\
2 & 1 & -3
\end{array}\right) \quad \hbox{and} \quad
Q^{\ast}=(q^{\ast}_{i,j})= \left(\begin{array}{ccc}
  -2 & 1 & 1\\
3 & -3 & 0\\
3 & 2 & -5
\end{array}\right).$$ For any $x \in
\R $, it is easy to see that
$$\barray
\ad q_{12}(x)+q_{13}(x)\le \bar{q}_{1,2}+\bar{q}_{1,3}, \quad q_{13}(x)\le \bar{q}_{1,3},\\
\ad q_{13}(x)\le \bar{q}_{2,3}, \quad q_{23}(x)\le \bar{q}_{2,3},\\
\ad q_{21}(x)\ge \bar{q}_{2,1}, \quad q_{21}(x)\ge \bar{q}_{3,1},\\
\ad q_{31}(x)+q_{32}(x)\ge \bar{q}_{3,1}+\bar{q}_{3,2}, \quad
q_{31}(x)\ge \bar{q}_{3,1}\earray$$ and that
$$\barray
\ad q^{\ast}_{1,2}+q^{\ast}_{1,3}\le q_{12}(x)+q_{13}(x),
\quad q^{\ast}_{1,3}\le q_{13}(x),\\
\ad q^{\ast}_{1,3}\le q_{23}(x), \quad q^{\ast}_{2,3}\le q_{23}(x),\\
\ad q^{\ast}_{2,1}\ge q_{21}(x), \quad q^{\ast}_{2,1}\ge q_{31}(x),\\
\ad q^{\ast}_{3,1}+q^{\ast}_{3,2}\ge q_{31}(x)+q_{32}(x), \quad
q^{\ast}_{3,1}\ge q_{31}(x).\earray$$ Hence, we get that $Q(x)
\preceq \bar{Q}$ and $Q^{\ast} \preceq Q(x)$ for all $x \in
\R $, it is easy to see that
$$\barray
\ad q_{12}(x)+q_{13}(x)\le q^{\ast}_{1,2}+q^{\ast}_{1,3}, \quad q_{13}(x)\le q^{\ast}_{1,3},\\
\ad q_{13}(x)\le q^{\ast}_{2,3}, \quad q_{23}(x)\le q^{\ast}_{2,3},\\
\ad q_{21}(x)\ge q^{\ast}_{2,1}, \quad q_{21}(x)\ge q^{\ast}_{3,1},\\
\ad q_{31}(x)+q_{32}(x)\ge q^{\ast}_{3,1}+q^{\ast}_{3,2}, \quad
q_{31}(x)\ge q^{\ast}_{3,1}\earray$$ and that
$$\barray
\ad q_{\ast}(1,2)+q_{\ast}(1,3)\le q_{12}(x)+q_{13}(x),
\quad q_{\ast}(1,3)\le q_{13}(x),\\
\ad q_{\ast}(1,3)\le q_{23}(x), \quad q_{\ast}(2,3)\le q_{23}(x),\\
\ad q_{\ast}(2,1)\ge q_{21}(x), \quad q_{\ast}(2,1)\ge q_{31}(x),\\
\ad q_{\ast}(3,1)+q_{\ast}(3,2)\ge q_{31}(x)+q_{32}(x), \quad
q_{\ast}(3,1)\ge q_{31}(x).\earray$$ Hence, we get that $Q(x)
\preceq Q^{\ast}$ and $Q_{\ast} \preceq Q(x)$ for all $x \in
\R^{1}$. Clearly, both $\bar{Q}$ and $Q^{\ast}$ are irreducible
and their associated invariant probability measures are given by
$\bar{\mu}=(\bar{\mu}_{1},\bar{\mu}_{2},\bar{\mu}_{3})
=({7}/{25},{8}/{25},{2}/{5})$
and
$\mu^{\ast}=(\mu^{\ast}_{1},\mu^{\ast}_{2},\mu^{\ast}_{3})
=({3}/{5},{7}/{25},{3}/{25})$ respectively. Thus, for the system $(X,\Lambda)$ satisfying (\ref{1.1}) and (\ref{1.2}) and any given
increasing function $h$ on $\S=\{1,2,3\}$, by virtue of
Lemma~\ref{OPC3}, we have
\beq{(OPC8b)}{\p}^{(x,i)}\biggl(\limsup_{t\to
\infty}\frac{1}{t}\int_{0}^{t}h(\Lambda(s))\d s \le
\frac{7h(1)+8h(2)+10h(3)}{25}\biggr)=1;\eeq and by virtue of
Lemma~\ref{OPC5}, we have
\beq{(OPC8b2)}{\p}^{(x,i)}\biggl(\liminf_{t\to
\infty}\frac{1}{t}\int_{0}^{t}h(\Lambda(s))\d s \ge
\frac{15h(1)+7h(2)+3h(3)}{25}\biggr)=1.\eeq}\end{Example}

\section{Almost sure stability of regime-switching SFDE}\label{Almost}

To make the idea clear, in this work we shall study the stability of a regime-switching system with linear feedback control as we did in \cite{Shao17a}. Recall the equation satisfied by $(X(t))$, i.e.
\begin{equation}\label{1.1a}
\d X(t)  =\big[a(X(t),\La(t))-b(\La(\de(t)))X(\de(t))\big]\d t+\sigma(X(t),\La(t))\d W(t),\ \
X(0)  =x \in \R^{d}.
\end{equation}
We would like to point out that these constants $b_i$, $i\in\S$ do not have to be all positive.
In order to show the almost sure stability of the system \eqref{1.1a} and \eqref{1.2},
we shall apply the method in \cite{Mao15}, and the key point is to prove
\[\int_0^\infty \E |X(t)|^2\d t<\infty.\]
As mentioned in \cite[Remark 3.3]{Shao17a}, we cannot show the finiteness of the following quantity at that time:
\[\E\int_0^\infty \e^{\int_0^t g(\La(\de(s)))\d s}\d t<\infty\]
for a general function $g$. However, with the help of Lemma \ref{t-2.1}, we can handle this quantity now. From now on, we suppose the coefficient $a(\cdot,\cdot):\R^d\times\S\ra \R^d$ satisfies the following condition.
\begin{itemize}
   \item[(\textbf{H3})] \ There exists a positive constant $M_a$ such that $|a(x,i)|\leq M_a |x|$ for all $(x,i)\in \R^d\times \S$.
\end{itemize}

To make our computation below more precisely, we give out a more explicit construction of the probability space used in the sequel.
Let
\[\Omega_1=\{\omega|\ \omega:[0,\infty)\ra \R^d \ \text{is continuous with $\omega(0)=0$}\},
\] which is endowed with the locally uniform convergence topology and the Wiener measure $\p_1$ so that the coordinate process $W(t,\omega):=\omega(t)$, $t\geq 0$, is a standard $d$-dimensional Brownian motion.
Let $(\Omega_2,\mathscr F_2,\p_2)$ be a probability space and $\Pi_{\R}$ be the totality of point functions on $\R$. For a point function $(p(t))$, $D_{p}$ denotes its domain, which is a countable subset of $[0,\infty)$. Let $p:\Omega_2\ra \Pi_{\R}$ be a Poisson point process with counting measure $N(\d t,\d z)$ on $(0,\infty)\times \R_{+}$ defined by
\begin{equation}
  \label{3.1} N((0,t)\times U)=\#\{s\in D_{p}|\ s\leq t,\ p(s)\in U\}, \quad  t>0, \quad U\in \mathscr B(\R_{+}),
\end{equation}
and its intensity measure is $\d t\times\m(\d z)$.
Set $(\Omega,\mathscr F,\p)=(\Omega_1\times\Omega_2,\mathscr B(\Omega_1)\times \mathscr F_2,\p_1\times\p_2)$, then under $\p=\p_1\times\p_2$, for $\omega=(\omega_1,\omega_2)$, $t\mapsto \omega_1(t)$ is a Wiener process, which is independent of the Poisson point process $t\mapsto p(t,\omega_2)$. Throughout this work, we will work on this probability space $(\Omega,\mathscr F,\p)$. Define \[\E^N[\,\cdot\,](\omega_2)=\E[\,\cdot\,|\mathscr F_2](\omega_2)\]
to be the conditional expectation with respect to the $\sigma$-algebra $\mathscr F_2$.

\begin{lem}\label{t3.1}
Let $(X(t),\La(t))$ be the solution of \eqref{1.1a} and \eqref{1.2}. Suppose conditions (\textbf{H1})-(\textbf{H3}), (\textbf{Q1}), (\textbf{Q2}) and (\textbf{Q3}) hold. Set
\begin{equation}\label{3.3}
K(\tau)=2\tau(2\bar C+ M_a+\bar b)\e^{(2\bar C+3M_a+\bar b)\tau},
\end{equation} where $\bar C=\max_{i\in\S} C(i)$ and $\bar b=\max_{i\in \S} b(i)$. Moreover, assume $\tau$ is sufficiently small so that $K(\tau)<1$. Then it holds
\begin{equation}\label{3.2}
\E^N|X(t)-X(\de(t))|^2(\omega_2)\leq \frac{K(\tau)}{1-K(\tau)}\E^N|X(t)|^2(\omega_2).
\end{equation}
\end{lem}

\begin{proof} For any $t\ge 0$, there is a unique integer $\upsilon\ge 0$ for $t\in [\upsilon\tau,(\upsilon+1)\tau)$. Moreover, $\de(s)=\upsilon\tau$ for
$s\in[\upsilon\tau,t]$. From \eqref{1.1a} we have that
\begin{align*}
    X(t)-X(\de(t))&=X(t)-X(\upsilon\tau)\\
    &=\int_{\upsilon\tau}^t \big[a(X(s),\La(s))-b(\La(\de(s)))X(\de(s))\big]\d s\\
    &\quad+\int_{\upsilon\tau}^t \sigma(X(s),\La(s))\d w_1(s).
  \end{align*}
  According to the condition (\textbf{H1}), it follows from this and It\^o's formula that
  \begin{align*}
    &|X(t)-X(\de(t))|^2\\
    &\quad=\int_{\de(t)}^t\!\! 2\la X(s)\!-\!X(\de(s)),a(X(s),\La(s))\!-\!b(\La(\de(s)))X(\de(s))\raa\d s\\
    &\qquad+\int_{\de(t)}^t\!\!\|\sigma(X(s),\La(s))\|_{\mathrm{HS}}^2\d s\!+\!\int_{\de(t)}^t\!\! 2\la X(s)\!-\!X(\de(s)),\sigma(X(s),\La(s))\d w_1(s)\raa.
  \end{align*}
  Taking the conditional expectation w.r.t. $\mathscr F_2$ on both sides of previous equality and using the independence of $(w_1(t))$ and $(w_2(t))$, we get
  \begin{align*}
    &\E^N[|X(t)-X(\de(t))|^2](\omega_2)\\
    &\leq \E^N\Big[\int_{\de(t)}^tC(\La(s))|X(s)|^2+2M_a|X(s)||X(s)-X(\de(s))|\d s\Big](\omega_2)\\
    &\quad+\E^N\Big[\int_{\de(t)}^t 2b(\La(\de(s)))|X(\de(s))||X(s)-X(\de(s))|\d s\Big](\omega_2)\\
    &\leq \E^N\Big[\int_{\de(t)}^t\big\{2\bar C(|X(s)-X(\de(s))|^2+|X(\de(s))|^2)\\
    &\quad +M_a(3|X(s)\!-\!X(\de(s))|^2\!+\!|X(\de(s))|^2)\\
    &\quad +\bar b|X(s)-X(\de(s))|^2+\bar b |X(\de(s))|^2 \big\}\d s\Big](\omega_2)\\
    &\leq (2\bar C+M_a+\bar b)\tau \E^N[|X(\de(t))|^2](\omega_2)\\
    &\quad +\E^N\Big[\int_{\de(t)}^t\!\!(2\bar C+3M_a+\bar b)|X(s)\!-\!X(\de(s))|^2\d s\Big](\omega_2).
  \end{align*}
  By virtue of Gronwall's inequality,
  \begin{align*}
    &\E^N|X(t)-X(\de(t))|^2(\omega_2)\\
    &\leq (2\bar C+ M_a+\bar b)\tau\E^N[|X(\de (t))|^2](\omega_2)\e^{\int_{\de(t)}^t\big(2\bar C+3M_a+\bar b\big)\d s}\\
    &\leq 2\tau(2\bar C+M_a+\bar b)\E^N[|X(t)-X(\de (t))|^2+|X(t)|^2](\omega_2)\e^{\int_{\de(t)}^t\big(2\bar C+3M_a+\bar b\big)\d s},
  \end{align*}
  which yields immediately the desired conclusion.
\end{proof}

In the following, we consider only the case $\S=\{1,2\}$ and use Lemma \ref{t-2.2} to construct the control Markov chains from upper and below. For the case $\S$ containing more than two states, we can use Lemma \ref{OPC3} and Lemma \ref{OPC5} to construct the control Markov chains, then follow the same line as the case $\S=\{1,2\}$ to derive the corresponding results.

Recall that $(\bar{q}_{ij})$ and $(q^\ast_{ij})$ are defined by \eqref{q-up} and \eqref{q-low}. Suppose that \eqref{2.6}, \eqref{2.7} and $\bar q_{21}$, $q_{12}^\ast>0$ hold.
Then, according to Lemma \ref{t-2.2}, it holds
\begin{equation}\label{comp}
\La^\ast(t)\leq \La(t)\leq \bar\La(t),\quad t\geq0, \quad \hbox{a.s.}
\end{equation}
Define
\begin{align*}
  \bar P_{ij}&=\p(\bar \La(\tau)=j|\bar \La(0)=i)\\
  P^\ast_{ij}&=\p(\La^\ast(\tau)=j|\La^\ast(0)=i),\qquad i,\,j\in \S.
\end{align*}
For a function $b(\cdot): \S\ra \R$, define
\[\bar P(b)=\big(\e^{b(i)}\bar P_{ij}\big),\quad P^\ast(b)=\big(\e^{b(i)}P^\ast_{ij}\big).\]
Then the corresponding first eigenvalues of the linear operators $\bar P_b$, $P_b^\ast$ are denoted by
\begin{equation}\label{eign}
\begin{split}
  \bar\lambda_{1}(b)&=\max\{\mathrm{Re}(\lambda);\lambda\in \mathrm{Spec}(\bar P_b)\},\\
  \lambda_{1}^\ast(b)&=\max\{\mathrm{Re}(\lambda);\lambda\in \mathrm{Spec}(P_b^\ast)\}.
\end{split}
\end{equation}

The Markov chain $(\bar\La(n\tau))_{n\geq 0}$ is a skeleton Markov chain of $(\bar \La(t))$. If we denote $\bar f_{ij}(t)=\p(\bar \La(t)=j|\bar\La(0)=i)$, $i,\,j\in\S$, $t\geq 0$, then
\[\bar P_{ij}=\p(\bar \La(\tau)=j|\bar \La(0)=i)=\bar f_{ij}(\tau).\]
It is known (cf. e.g. \cite[Chapter 4]{C2004}) that $\bar f_{ij}(t)$ satisfies the following equation
\begin{equation}\label{dens-1}
\bar f_{ij}(t)=\e^{\bar q_i t} \delta_{ij}+\sum_{k\neq j}\int_0^t \bar f_{ik}(t-s)\bar q_{kj}\e^{-\bar q_js}\d s,
\end{equation}
where $\delta_{ij}=1$ if $i=j$; otherwise, $\delta_{ij}=0$. Since $(\bar \La(t))$ is assumed to be irreducible, this yields that $\bar P_{ij}>0$ for all $i,\,j\in\S$, which means that the transition matrix $\bar P$ of $(\bar\La(n\tau))_{n\geq 0}$ is positive. Therefore, Lemma \ref{t-2.1} can be applied to $(\bar\La(n\tau))_{n\geq 0}$. Similar deduction holds for $(\La^\ast(n\tau))_{n\geq 0}$.

\begin{lem}\label{t3.2}
There exist two   constants $\wt K_1$, $\wt K_2>0$ such that for any initial point $i_0\in\S$ of $\bar \La$ and $\La^\ast$, it holds that
\begin{gather}\label{3.5}
  \wt K_1 (\bar \lambda_{1}(b))^t\leq \E_{i_0}\Big[\exp\Big\{\int_0^tb(\bar \La(\de(s)))\d s\Big\}\Big]\leq \wt K_2(\bar \lambda_{1}(b))^t,\\
  \wt K_1 (\lambda_{1}^\ast(b))^t\leq \E_{i_0}\Big[\exp\Big\{\int_0^tb(\La^\ast(\de(s)))\d s\Big\}\Big]\leq \wt K_2(\lambda_{1}^\ast(b))^t,
\end{gather} for $t$ large enough.
\end{lem}

\begin{proof}
We only prove \eqref{3.5}, and the rest assertion can be proved in the same way.
Applying Lemma \ref{t-2.1}, we obtain that for $t$ large enough,
\begin{align*}
  \E_{i_0}\Big[\exp\Big\{\int_0^t b(\bar \La(\de (s)))\d s\Big\}\Big]
  &=\E_{i_0}\Big[\exp\Big\{\sum_{k=0}^{[t/\tau]-1}\big(b(\bar \La(k\tau))\tau+(t-[t/\tau]\tau)\big)\Big\}\Big]\\
  &\leq \e^{\tau\max_{i}b(i)}\E_{i_0}
  \Big[\exp\Big\{\sum_{k=0}^{[t/\tau]-1}b(\bar\La(k\tau))\tau\Big\}\Big]\\
  &\leq \e^{\tau\max_{i}b(i)} K_2\big(\min\{\bar \La_{1,b},1\}\big)^{-\tau}(\lambda_{1}(b))^t.
\end{align*}
Analogously,
\begin{align*}
  \E_{i_0}\Big[\exp\Big\{\int_0^t b(\bar \La(\de (s)))\d s\Big\}\Big]
  &\geq \e^{\tau(\min_{i}b(i)\wedge 0)}K_1\big(\max\{\bar \lambda_{1,b},1\})^{-\tau}(\bar\lambda_{1}(b))^t.
\end{align*}
Inequality \eqref{3.5} follows from the finiteness of the cardinality of $\S$.
\end{proof}

\begin{lem}\label{t3.3}
Under the conditions (\textbf{H1})-(\textbf{H3}), (\textbf{Q1}), (\textbf{Q2}) and (\textbf{Q3}), suppose further that \eqref{2.6}, \eqref{2.7} and $\bar q_{21}$, $q_{12}^\ast>0$ hold. Assume the functions $b(\cdot)$, $C(\cdot)$, $c(\cdot)$ on $\S$ are all non-decreasing. Then for $\veps\in(0,1)$,
\begin{equation}\label{3.8}
\E[|X(t)|^2]\leq |x_0|^2\E\Big[\e^{\int_0^t (C(\bar\La(r))\!-\!2b(\La^\ast(\de(r)))
    \!+\!2\sqrt{\frac{K(\tau)}{1-K(\tau)}}b(\bar\La(\de(r))))\d r}\Big]
\end{equation}
and
\begin{equation}\label{3.9}
\E[|X(t)|^2]\geq |x_0|^2\E\Big[\e^{\int_0^t
c(\La^\ast(r))\!-\!2b(\bar\La(\de(r)))
\!-\!2\sqrt{\frac{K(\tau)}{1-K(\tau)}}b(\bar\La(\de(r)))\d r}\Big].
\end{equation}
\end{lem}

\begin{proof}
  It follows from It\^o's formula and condition (\textbf{H1}) that
  \begin{align*}
    \d |X(t)|^2
    &=\big(2\la X(t),a(X(t),\La(t))\raa+
    \|\sigma(X(t),\La(t))\|_{\mathrm{HS}}^2\big)\d t\\
    &\quad-2b(\La(\de(t)))\la X(t),X(\de(t))\raa \d t+
    2\la X(t),\sg(X(t),\La(t))\d w_1(t)\raa\\
    &\leq\big(C(\La(t))|X(t)|^2-2b(\La(\de(t)))
    \la X(t),X(\de(t))\raa\big)\d t\\
    &\quad+
    2\la X(t),\sg(X(t),\La(t))\d w_1(t)\raa\\
    &\leq\big(C(\La(t))-(2-\veps)b(\La(\de(t)))\big)|X(t)|^2\d t\\
    &\quad+\frac 1\veps b(\La(\de(t)))|X(t)-X(\de(t))|^2\d t\\
    &\quad+2\la X(t),\sigma(X(t),\La(t))\d w_1(t)\raa, \, \, \hbox{for every} \, \, \veps\in(0,1).
  \end{align*}
  Taking conditional expectation $\E^N[\,\cdot\,]$ on both sides, then applying \eqref{comp} and the non-decreasing property of functions $b(\cdot)$ and $C(\cdot)$, we obtain that for $0\leq s<t$,
  \begin{align*}
    &\E^N[|X(t)|^2](\omega_2)-\E^N[|X(s)|^2](\omega_2)\\
    &\leq \E^N\Big[\int_s^t\Big\{ \big(C(\La(r))-(2-\veps)b(\La(\de(r)))\big)|X(r)|^2\\
     &\qquad\qquad+\frac 1\veps b(\La(\de(r)))|X(r)-X(\de(r))|^2\Big\}\d r\Big](\omega_2)\\
     &\leq \E^N\Big[\int_s^t\Big\{ \big(C(\bar\La(r))-(2-\veps)b(\La^\ast(\de(r)))\big)|X(r)|^2\\
     &\qquad\qquad+\frac1 \veps b(\bar\La(\de(r)))|X(r)-X(\de(r))|^2\Big\}\d r \Big](\omega_2)\\
     &=\int_s^t\Big\{\big(C(\bar\La(r))-(2-\veps)b(\La^\ast(\de(r)))\big)
     \E^N\big[|X(r)|^2\big](\omega_2)\\
     &\qquad \qquad \quad +\frac1 \veps b(\bar\La(\de(r)))\E^N\big[|X(r)-X(\de(r))|^2\big](\omega_2)\Big\}\d r,
  \end{align*} where in the last equality we used the fact that the processes $(\bar\La(t))$ and $(\La^\ast(t))$ are fixed once $\omega_2$ is given. Invoking the estimate in Lemma \ref{t3.1}, this yields
  \begin{equation}\label{3.10-1}
  \begin{split}
    &\E^N[|X(t)|^2](\omega_2)-\E^N[|X(s)|^2](\omega_2)\\
    &\leq \int_s^t\Big\{ \
    \Big(C(\bar\La(r))\!-\!(2-\veps)b(\La^\ast(\de(r)))
    \!+\!\frac{K(\tau)}{\veps (1-K(\tau))}b(\bar\La(\de(r)))\Big)
    \E^N[|X(r)|^2](\omega_2)\Big\}\d r.
  \end{split}
  \end{equation} Note that as a function of $\veps$,
  $$C(\bar\La(r))\!-\!(2-\veps)b(\La^\ast(\de(r)))
    \!+\!\frac{K(\tau)}{\veps (1-K(\tau))}b(\bar\La(\de(r)))$$
    takes its maximal value at
  $$\veps =\sqrt{\frac{K(\tau)}{1-K(\tau)}} \cdot \sqrt{\frac{b(\bar\La(\de(r)))}{b(\La^\ast(\de(r)))}}.$$
Hence, it follows from \eqref{3.10-1} that
  \begin{align*}
    &\E^N[|X(t)|^2](\omega_2)-\E^N[|X(s)|^2](\omega_2)\\
    &\leq \!\int_s^t\!\Big\{
    \Big(C(\bar\La(r))\!-\!2b(\La^\ast(\de(r)))
    \!+\!2\sqrt{\frac{K(\tau)}{1-K(\tau)}}
    \sqrt{b(\bar\La(\de(r)))b(\La^\ast(\de(r)))}\Big)
    \E^N[|X(r)|^2](\omega_2)\Big\}\d r.
 \end{align*}
  Since $b(\La^\ast(\de(r)))\leq b(\bar\La(\de(r)))$ for all $r\ge 0$, we then have that
  \begin{equation}\label{3.10}
  \begin{split}
    &\E^N[|X(t)|^2](\omega_2)-\E^N[|X(s)|^2](\omega_2)\\
    &\leq \int_s^t\Big\{ \
    \Big(C(\bar\La(r))\!-\!2b(\La^\ast(\de(r)))
    \!+\!2\sqrt{\frac{K(\tau)}{1-K(\tau)}}b(\bar\La(\de(r)))\Big)
    \E^N[|X(r)|^2](\omega_2)\Big\}\d r.
  \end{split}
  \end{equation}
  Set $u(t)(\omega_2)=\E^N[|X(t)|^2](\omega_2)$, and
  \[g(r)=C(\bar\La(r))\!-\!2b(\La^\ast(\de(r)))
    \!+\!2\sqrt{\frac{K(\tau)}{1-K(\tau)}}b(\bar\La(\de(r))).\]
  \eqref{3.10} can be rewritten in the form
  \begin{equation}\label{3.11}
  u(t)(\omega_2)-u(s)(\omega_2)\leq \int_s^t g(r)u(r)(\omega_2)\d r.
  \end{equation}
  As the function $r\ra g(r)$ needs not to be differentiable, we can use the trick as in \cite[Lemma 3.2]{Shao17} to derive from Gronwall's inequality that
  \begin{equation}\label{3.12}
  u(t)(\omega_2)\leq u(0)(\omega_2)\e^{\int_0^t g(r)\d r}.
  \end{equation}
  Then we get the desired upper estimate \eqref{3.8} after taking expectation w.r.t. $\E[\cdot]$ on both sides of \eqref{3.12}.

  The lower estimate \eqref{3.9} can be deduced by the same method. Actually, it follows from It\^o's formula and condition (\textbf{H1}) that
  \begin{align*}
    \d |X(t)|^2
    &=\big(2\la X(t),a(X(t),\La(t))\raa+
    \|\sigma(X(t),\La(t))\|_{\mathrm{HS}}^2\big)\d t\\
    &\quad-2b(\La(\de(t)))\la X(t),X(\de(t))\raa \d t+
    2\la X(t),\sg(X(t),\La(t))\d w_1(t)\raa\\
    &\ge\big(c(\La(t))-(2+\veps)b(\La(\de(t)))\big)|X(t)|^2\d t\\
    &\quad-\frac 1\veps b(\La(\de(t)))|X(t)-X(\de(t))|^2\d t\\
    &\quad+2\la X(t),\sigma(X(t),\La(t))\d w_1(t)\raa,\, \, \hbox{for every} \, \, \veps\in(0,1).
  \end{align*} In what follows, the difference is that we shall use the transform $c(\La(r))\geq c(\La^\ast(r))$ instead of $C(\La(r))\leq C(\bar\La(r))$ and corresponding transform for $b(\La(r))$ in this case. However, these details are omitted.
\end{proof}

In order to estimate the long time behavior of the quantity
\[\E\e^{p\!\int_0^t C(\bar\La(r))\d r},\qquad p>0,\]
we present the estimate established in \cite{Bar} after introducing some necessary notation.
Let
\[\bar Q_p=\bar Q+p\mathrm{diag}(C(1),\ldots,C(M)),\]
where $\mathrm{diag}(C(1),\ldots,C(M))$ denotes the diagonal matrix generated by  vector $(C(1), \ldots, C(M))$. Put
\begin{equation}\label{3.13}
\eta_{p,C}=-\max\{\mathrm{Re}(\gamma);\,\gamma\in \mathrm{Spec}(\bar Q_p)\}.
\end{equation}
According to \cite[Proposition 4.1]{Bar}, for any $p>0$, there exist two positive constants $\kappa_1(p)$, $\kappa_2(p)$ such that for any initial point $i_0\in\S$,
\begin{equation}\label{3.14}
\kappa_1(p)\e^{-\eta_{p,C}t}\leq \E_{i_0}\Big[\e^{p\!\int_0^t C(\bar\La(r))\d r}\Big]\leq \kappa_2(p)\e^{-\eta_{p,C}t},\quad t>0.
\end{equation}

Now we formulate our main result.

\begin{thm}\label{t3.4}
Suppose the conditions in Lemma \ref{t3.3} hold. In addition, assume
\begin{equation}\label{3.15}
\eta_{3,C}>0,\ \ \lambda_{1}^{\ast}(-6b)<1, \ \ \text{and}\ \ \bar\lambda_{1}\Big(6\sqrt{\frac{K(\tau)}{1-K(\tau)}}b\Big)<1.
\end{equation}
Then for every initial value $ (X(0),\La(0))=(x_0,i_0)\in \R^d\times\S$ of \eqref{1.1} and \eqref{1.2},  it holds
\begin{equation}\label{3.16}
\lim_{t\ra\infty} X(t)=0,\ \ \text{a.s.}
\end{equation}
\end{thm}

\begin{proof}
  Applying Lemma \ref{t3.2}, \eqref{3.14} and H\"older's inequality to the estimate \eqref{3.8}, we get
  \begin{equation}\label{3.17}
  \begin{split}
    \E[|X(t)|^2]&\leq |x_0|^2 \Big(\E\e^{3\int_0^t C(\bar\La(r))\d r}\Big)^{\frac 13}\Big(\E\e^{-6\int_0^t b(\La^\ast(\de(r)))\d r}\Big)^{\frac 13}\Big(\E\e^{6\sqrt{\frac{K(\tau)}{1-K(\tau)}}
    \int_0^t b(\bar\La(\de(r)))\d r}\Big)^{\frac 13}\\
    &\leq |x_0|^2 \big(\kappa_2(3)\wt K_2^2\big)^{\frac 13} \e^{-\frac{\eta_{3,C}t}{3}}\big( \lambda_{1}^{\ast}(-6b)\big)^{\frac t3}\Big(\bar \lambda_{1}\Big(6\sqrt{\frac{K(\tau)}{1-K(\tau)}}b
    \Big)\Big)^{\frac t3}.
  \end{split}
  \end{equation}
  Then, it is easy to check that  under the condition \eqref{3.15},
\begin{equation}\label{3.18a}
\int_0^\infty\E\big[|X(t)|^2\big]\d t<\infty,
\end{equation}
and there exists a constant $C>0$ such that
\beq{C}
\E[|X(t)|^2]\le C \quad \hbox{for all} \quad t\ge 0.
\eeq
By It\^o's formula, we obtain that
  \begin{align*}
    &\E[|X(t_2)|^2]-\E[|X(t_1)|^2]\\
    &\quad=\E\int_{t_1}^{t_2}\!\! \big(2\la X(s),a(X(s),\La(s))\!-\!b(\La(\de(s)))X(\de(s))\raa\\
    &\qquad \qquad \qquad +\|\sigma(X(s),\La(s))\|_{\mathrm{HS}}^2\big)\d s
  \end{align*} for any $0\le t_1 <t_2<\infty$. Thus, by virtue of the condition (\textbf{H1}) and \eqref{C}, there exists a generic constant $C>0$ such that
  $$\big|\E\big[|X(t_2)|^2]-\E[|X(t_1)|^2\big]\big|\le C(t_2-t_1).$$
  Namely, $\E[|X(t)|^2]$ is uniformly continuous with respect to $t$ over $\R_{+}$. Hence, it follows from \eqref{3.18a} that
  \begin{equation}\label{3.20}
  \lim_{t\ra \infty}\E\big[|X(t)|^2\big]=0.
  \end{equation}
  Now we can completely follow the proof line of \cite[Theorem 3.4]{Mao15} to show that
  \[\lim_{t\ra \infty} X(t)=0,\ \ \text{a.s.}\]
  The details are omitted.
\end{proof}
\begin{rem}
  Note that in Lemma \ref{t3.2} and Theorem \ref{t3.4}, we have assumed the non-decreasing property of the functions $b(\cdot)$, $C(\cdot)$, $c(\cdot)$ on $\S$. This is a technical assumption to simplify our presentation. Without this monotone condition, after doing some necessary rearrangement of the states of $\S$, our results remain valid. Precisely, for instance, in order to control $\exp\big(\int_0^t C(\La(s))\d s\big)$, one can first reorder the set $\S$ so that $C(\cdot)$ is non-decreasing. Of course, under this new order, the original $Q$ matrix becomes a new form $\tilde Q=(\tilde q_{ij})$, while $\tilde Q$ remains to be conservative and irreducible.  Hence, we can define the corresponding Markov chains $(\bar \La(t))$ and $(\La^\ast(t))$ associated with $\tilde Q$. Then, applying Lemmas \ref{t-2.2}, \ref{OPC3}, \ref{OPC5},  one can control  $\exp\big(\int_0^t C(\La(s))\d s\big)$ from upper and below.
\end{rem}

\section*{Appendix A: Proof of Theorem \ref{exiuni}}\label{appendix}

According to \cite{Sk89} and \cite{YZ}, $(\La(t))$ can be represented in terms of Poisson random measure. This representation will play an important role in this work.
For the sake of clarity in the presentation and calculation, we introduce the following construction of the probability space which will be used throughout this work. Let
\begin{equation*}
  \Omega^{(1)}=\{\omega|\,\omega:[0,\infty)\ra \R^d\ \text{is continuous with $\omega_0=0$}\},
\end{equation*} which is endowed with the locally uniform convergence topology and the Wiener measure $\p^{(1)}$ so that the coordinate process $W(t,\omega):=\omega(t), \ t\geq 0$ is a standard $d$-dimensional Wiener process on $(\Omega^{(1)}, {\mathcal{F}}^{(1)}, \{{\mathcal{F}}^{(1)}_{t}\}_{t\ge 0}, \, \p^{(1)})$.
Let $(\Omega^{(2)}, {\mathcal{F}}^{(2)}, \{{\mathcal{F}}^{(2)}_{t}\}_{t\ge 0}, \, \p^{(2)})$ be a
complete probability space with a filtration $\{{\mathcal{F}}^{(2)}_{t}\}_{t\ge
0}$ satisfying the usual conditions,
and let $\{\xi_{n}\}$ be a sequence of independent exponentially distributed random variables with mean $1$ on $(\Omega^{(2)}, {\mathcal{F}}^{(2)}, \{{\mathcal{F}}^{(2)}_{t}\}_{t\ge 0}, \, \p^{(2)})$. Define
\[\Omega=\Omega^{(1)}\times\Omega^{(2)},\quad {\mathcal{F}}={\mathcal{F}}^{(1)}\times {\mathcal{F}}^{(2)},\quad
{\mathcal{F}}_{t}={\mathcal{F}}^{(1)}_{t}\times {\mathcal{F}}^{(2)}_{t},\quad \p=\p^{(1)}\times\p^{(2)},\]
and  $(\Omega, {\mathcal{F}}, \{{\mathcal{F}}_{t}\}_{t\ge 0}, \, \p)$
is just the probability space used throughout this appendix. The proof of Theorem \ref{exiuni} is a little long, so we separate it into two steps.

\noindent\textbf{Step 1: Construction of solution.}\
Fix some $(x,i)\in \R^{d}\times \S$ and consider
 the following SFDE
 \begin{equation}\label{1.1i}
\d X^{(i)}(t)=\big[a(X^{(i)}(t),i)-b(X^{(i)}(\de(t)),i)\big]\d t+\sigma(X^{(i)}(t),i)\d W(t)
\end{equation} with $X^{(i)}(0)=x$. By virtue of conditions (\textbf{H1}) and (\textbf{H2}), we can prove that equation \eqref{1.1i} admits a unique nonexplosive solution $X^{(i)}(t)$ by using the Picard iterations following the line of \cite[Chapter 5, Theorem 2.2]{M1997}. Then, we have
\[\p\big(\lim_{m\to\infty}\tau_{O_{m}}=\infty\big)=1,\] where $\tau_{O_{m}}:=\inf\{t\ge 0: |X^{(i)}(t)|\ge m\}$ for $m\ge 1$.
Recall that $\{\xi_{n}\}$ is a sequence of independent mean $1$ exponentially distributed random variables introduced above. Let \begin{equation}
\label{eq-tau1}
 \tau_{1}= \theta_{1}: = \inf\Big\{t\ge 0: \int_{0}^{t} q_{i}(X^{(i)}(s))\d s > \xi_{1}\Big\},
\end{equation} so we have \begin{equation}\label{eq1-sw-mechanism}
\p\big(\tau_{1} > t| \mathcal{F}^{(1)}_{t}\big) = \p\Big(\xi_{1} \ge  \int_{0}^{t} q_{i}(X^{(i)}(s))\d s\Big| \mathcal{F}^{(1)}_{t}\Big)= \exp\Big\{-\int_{0}^{t} q_{i}(X^{(i)}(s))\d s\Big\}.
\end{equation} Then, it follows from condition (\textbf{Q2'}) that for some $m\ge |x|+1$, \beq{tau1zero1}\begin{array}{ll} \p\big(\tau_{1}>t\big)\ad
=\E \exp\Big\{-\int_{0}^{t} q_{i}(X^{(i)}(s))\d s\Big\}\\
\ad \ge \E\Big(\one_{\{\tau_{O_{m}}\ge t\}} \exp\Big\{-\int_{0}^{t} q_{i}(X^{(i)}(s))\d s\Big\}\Big)\\
\ad =\E\Big(\one_{\{\tau_{O_{m}}\ge t\}} \exp\Big\{-\int_{0}^{t\wedge \tau_{O_{m}}} q_{i}(X^{(i)}(s))\d s\Big\}\Big)\\
\ad \ge \p\big(\{\tau_{O_{m}}\ge t\}\big) \exp\Big\{-K_{0}\big(1+m^{\kappa_{0}}\big)t\Big\}.
\end{array}\eeq Since both terms of the product on the last line tend to $1$ as $t\downarrow 0$, one gets $\p\big(\tau_{1}>0\big)=1$. We define a process $(X,\La) \in \R^{d}\times \S$ on $[0, \tau_{1}]$ as follows:
\begin{equation}\label{step-1}
 X(t) = X^{(i)}(t) \text{ for all } t \in [0, \tau_{1}], \text{ and }\La(t)  =    i   \text{ for all } t \in [0, \tau_{1}).
\end{equation}
   Moreover, we define $\La(\tau_{1})\in \S$ according to the probability distribution: \begin{equation}\label{la-1}
 \p\big(\La(\tau_{1}) = j| {\mathcal{F}}_{\tau_{1}-}\big) = \dfrac{q_{ij}(X(\tau_{1}))}{q_{i}(X(\tau_{1}))} (1- \delta_{ij}) \one_{\{q_{i}(X(\tau_{1})) >0\}} + \delta_{ij} \one_{\{q_{i}(X(\tau_{1})) =0\}}.
 \end{equation}

Obviously, the process $(X,\La)$ has been constructed on the temporal interval $[0,\tau_{1}]$, so the process $(X,\La)$ is well-defined at the observation time $\de(t)$ when $\de(t)\le \tau_{1}$. Next, we construct the process $(X,\La)$ after $\tau_{1}$. To do so, when $\de(t)<\tau_{1}$, let $\wdh{X}$ satisfy
\beq{hatXbeforetau1}
\d \wdh{X}(t)
=\big[a(\wdh{X}(t),\La(\tau_{1}))-b(X^{(i)}(\de(t+\tau_{1})),i)\big]\d t+\sigma(\wdh{X}(t),\La(\tau_{1}))\d \wdt{W}(t),
\eeq with $\wdh{X}(0)=X(\tau_{1})$; when $\de(t)\ge\tau_{1}$, let $\wdh{X}$ satisfy
\beq{hatXaftertau1}
\d \wdh{X}(t)
=\big[a(\wdh{X}(t),\La(\tau_{1}))-b(\wdh{X}(\de(t)),\La(\tau_{1}))\big]\d t+\sigma(\wdh{X}(t),\La(\tau_{1}))\d \wdt{W}(t)
\eeq with $\wdh{X}(0)=X(\tau_{1})$, where $\wdt{W}(t)=W(t+\tau_{1})-W(\tau_{1})$. Actually, we can combining the above two equations \eqref{hatXbeforetau1} and \eqref{hatXaftertau1} as the following SFDE:
\beq{hatX}
\d \wdh{X}(t)
=\big[a(\wdh{X}(t),\La(\tau_{1}))-b(\wdh{X}(\de(t)),\wdh{\La}(\de(t)))\big]\d t+\sigma(\wdh{X}(t),\La(\tau_{1}))\d \wdt{W}(t),
\eeq with $\wdh{X}(0)=X(\tau_{1})$, where
\beq{wdhXde(t)}\wdh{X}(\de(t))=X^{(i)}(\de(t+\tau_{1}))\one_{\{0\le \de(t)<\tau_{1}\}}+\wdh{X}(\de(t))\one_{\{\de(t)\ge \tau_{1}\}},\eeq
\beq{wdhLade(t)}\wdh{\La}(\de(t))=i\one_{\{0\le \de(t)<\tau_{1}\}}+\La(\tau_{1})\one_{\{\de(t)\ge \tau_{1}\}},\eeq and here, $X^{(i)}$ is the unique solution to equation \eqref{1.1i} and so $X^{(i)}(\de(t+\tau_{1}))$ is well-defined.

Clearly, it is easy to see from \eqref{wdhXde(t)} and \eqref{wdhLade(t)} that $\wdh{X}(\de(t))$ is well-defined whenever $0\le \de(t)<\tau_{1}$, while $\wdh{\La}(\de(t))$ is well-defined for both $\de(t)<\tau_{1}$ and $\de(t)\ge\tau_{1}$. Therefore, equation \eqref{hatX} has a unique nonexplosive solution $\wdh{X}(t)$ thanks to conditions (\textbf{H1}) and (\textbf{H2}). Let \begin{equation}
 \label{eq-theta-2}
\theta_{2}:=\inf\Big\{t\ge 0: \int_{0}^{t} q_{\La(\tau_{1})}(\wdh{X}(s))\d s > \xi_{2}\biggr\}. \end{equation} As argued in \eqref{eq1-sw-mechanism}, we have
\begin{equation*}\label{eq-theta-2-distribution} \begin{aligned}
\p\big(\theta_{2} > t| {\mathcal{F}}_{\tau_{1}+t}\big) & = \p\Big(\xi_{2} \ge \int_{0}^{t} q_{\La(\tau_{1})}(\wdh{X}(s))\d s\Big|\F_{\tau_{1}+t}\Big)  \\ & = \exp\Big\{-\int_{0}^{t} q_{\La(\tau_{1})}(\wdh{X}(s))\d s\Big\}.
\end{aligned}\end{equation*}
Once again, we can derive from condition (\textbf{Q2'}) that $\p\big(\theta_{2} >0\big)=1$. Then we let \begin{equation}
\label{eq-tau-2}
 \tau_{2} : = \tau_{1} + \theta_{2}= \theta_{1} + \theta_{2}
\end{equation}  and define $(X,\La)$ on $[\tau_{1}, \tau_{2}]$ by
\begin{align}\label{step-2}
X(t)=\wdh{X}(t-\tau_{1})  \text{ for } t\in  [\tau_{1}, \tau_{2}], \ \ \La(t) = \La(\tau_{1})  \text{ for } t\in  [\tau_{1}, \tau_{2}), \
\end{align}
and
\begin{equation}\label{la-2}\begin{aligned}  \p& \big(\La(\tau_{2}) = l| {\mathcal{F}}_{\tau_{2}-}\big)  \\  & = \dfrac{q_{\La(\tau_1),l}(X(\tau_{2}))}{q_{\La(\tau_1)}(X(\tau_{2}))} (1- \delta_{\La(\tau_1),l})  \one_{\{q_{\La(\tau_1)}(X(\tau_{2})) > 0  \}} + \delta_{\La(\tau_1),l} \one_{\{q_{\La(\tau_1)}(X(\tau_{2})) = 0  \}}.
\end{aligned} \end{equation}

Following this procedure, we can further define $(X,\La)$ on the interval $[\tau_{n}, \tau_{n+1})$ inductively for $n\geq 3$, where $\tau_{n}$ is defined similarly as \eqref{eq-tau1} and \eqref{eq-tau-2}. This ``interlacing procedure'' uniquely determines a process $(X,\La)\in \R^{d}\times \S$ for all $t \in[0,\tau_{\infty})$, where
 \begin{equation}
\label{eq-tau-infty-defn}
 \tau_{\infty}=\lim_{n\to \infty}\tau_n.
\end{equation} Since the sequence $\tau_{n}$ is strictly increasing, the limit $\tau_{\infty} \le \infty$ exists. Hence, the process $(X,\La)$ constructed above can be regarded as the unique solution to SFDE \eqref{1.1b} and \eqref{1.2} on $[0,\tau_{\infty})$.

\noindent\textbf{Step 2: Nonexplosion of solution.}\
What is left to complete the proof of Theorem \ref{exiuni} is to show $\p(\tau_\infty=\infty)=1$, which is also the most delicate and difficult part of the argument.

\emph{First}, we show that the evolution of
the discrete component $\Lambda$ can be represented as a stochastic
integral with respect to a Poisson random measure, which is sometimes called Skorokhod's representation of $\Lambda$. In view of \cite[Section II-2.1]{Sk89} or \cite[Section 2.2]{YZ}, for each $x\in \R^d$, construct a family of intervals $\{\Gamma_{ij}(x);\ i,j\in \S\}$  on the positive half line in the following manner:
\begin{align*}
  \Gamma_{12}(x)&=[0,q_{12}(x))\\
  \Gamma_{13}(x)&=[q_{12}(x),q_{12}(x)+q_{13}(x))\\
  &\ldots\\
  \Gamma_{1M}(x)&=\Big[\sum_{j=2}^{M-1}q_{1j}(x), q_{1}(x)\Big)\\
  \Gamma_{21}(x)&=[q_1(x),q_1(x)+q_{21}(x))\\
  \Gamma_{23}(x)&=[q_1(x)+q_{21}(x),q_1(x)+q_{21}(x)+q_{23}(x))\\
  &\ldots
\end{align*}
and so on. Therefore, we obtain a sequence of consecutive, left-closed, right-open intervals $\Gamma_{ij}(x)$, each having length $q_{ij}(x)$. For convenience of notation, we set $\Gamma_{ii}(x)=\emptyset$ and $\Gamma_{ij}(x)=\emptyset$ if $q_{ij}(x)=0$.
Define a function $h:\R^d\times \S\times \R_{+}\ra \R$ by
\[h(x,i,z)=\sum_{j\in \S} (j-i)\mathbf 1_{\Gamma_{ij}(x)}(z).\]
Namely, for each $x\in \R^{d}$ and $i\in \S$, we set $h(x,i,z)=j-i$ if $z \in\Gamma_{ij}(z)$ for some $j\neq i$; otherwise $h(x,i,z)=0$.

Put $\lambda(t) : = \int_{0}^{t}q_{\La(s)} (X(s))\d s$ and
$n(t): = \max\{n \in \mathbb N: \xi_{1} + \dots+ \xi_{n} \le \lambda(t) \}$
for all $t \in [0, \tau_{\infty})$, where $\{ \xi_{n}, n=1,2,\dots\}$
is the sequence of independent exponential random variables with mean 1 introduced above.
Then in view of \eqref{eq-tau1}, \eqref{eq1-sw-mechanism},
\eqref{eq-theta-2}, and \eqref{eq-tau-2}, the process $\{n(t\wedge\tau_{\infty}), t\ge 0\}$
is a counting process that counts the number of switches for the component $\La$. We can regard $n(\cdot) $ as a nonhomogeneous Poisson process with random intensity function $q_{\La(t)}(X(t))$, $t \in [0, \tau_{\infty})$.

Now for any $s < t \in   [0, \tau_{\infty})$ and  $A\in \B(\S)$, let
$$ p((s,t]\times A)=\sum_{u \in (s, t]}\one_{\{\La(u)\neq \La(u-),
\La(u)\in A\}} \text{ and }  p(t,A)=  p((0,t]\times A).$$
Then we have $ p(t\wedge\tau_{\infty}, \S)=n(t\wedge\tau_{\infty})$ and
\begin{equation}
\label{eq0-La-sde}\begin{aligned}
 \La(t\wedge\tau_{\infty}) & =  \La(0) + \sum_{k=1}^{\infty}[\La(\tau_{k})- \La(\tau_{k}-)] \one_{\{\tau_{k}\le t\wedge\tau_{\infty}\}} \\ & = \La(0) + \int_{0}^{t\wedge\tau_{\infty}} \int_{\S}[j- \La(s-)] \,  p (\d s, \d j).
\end{aligned}\end{equation} We can also define a Poisson random measure $N(\cdot,\cdot)$ on $[0, \infty) \times \R_{+}$ by  $$N(t\wedge\tau_{\infty}, B):= \sum_{j\in \S \cap B} p(t\wedge\tau_{\infty},j), \ \text{ for all } t \ge 0 \text{ and } B \in \mathcal B(\R_{+}).$$
Observe that for any $(x,i)\in \R^{d}\times \S$ and $j \in \S\setminus\{i\}$, we have \begin{displaymath}
 \m\{z\in [0,\infty): h(x,i,z) \neq 0 \} = q_{i}(x) \text{ and }\m\{z\in [0,\infty): h(x,i,z)=j-i\}=q_{ij}(x),
\end{displaymath} where $\m$ is the Lebesgue measure on $\R_{+}$. Therefore, we can rewrite \eqref{la-1}, \eqref{la-2} and \eqref{eq0-La-sde} into the following form
\begin{equation}\label{eq:La-SDE}
 \La(t\wedge\tau_{\infty})=\La(0) + \int_{0}^{t\wedge\tau_{\infty}}\int_{\R_{+}}h(X(s-),\La(s-),z){N}(\d s,\d z).
\end{equation}

\emph{Second}, we shall prove that the process $X$ is nonexplosive. Without loss of generality, we fix the initial value $(X(0),\La(0)) = (x,i)\in \R^{d}\times \S$  and for any integer $m \ge [|x|]+1$, denote by $\wdt{\tau}_m :=\inf \{t\ge 0: |X(t)|\ge m\}$ the first exit time for the $X$ component from the open ball $O_{m}:=\{x\in\R^{d}: |x|<m\}$, and let $\wdt\tau_{\infty}: =\lim_{m\to \infty}\wdt \tau_{m}$. We shall prove that
\beq{inftyone}\p\big(\wdt\tau_{\infty}=\infty\big)=1.\eeq To this end, we first consider the $X$ component process on the temporal interval $[0,\tau)$, where $\tau>0$ is the length of discrete time observation period.  Actually, by It\^o's formula, we have
\begin{equation}\label{2.9A}
  \begin{split}
    \d |X(t)|^2&=2\la X(t),a(X(t),\La(t))-b(X(\de(t)),\La(\de(t)))\raa\d t\\ &\quad +\norm{\sigma(X(t),\La(t))}_{\mathrm{HS}}^2\d t+2\la X(t), \sigma(X(t),\La(t))\d W(t)\raa\\
    &=\big(2\la X(t),a(X(t),\La(t))\raa+\norm{\sigma(X(t),\La(t))}_{\mathrm{HS}}^2\big)\d t\\
    &\quad -2\la X(t),b(X(\de(t)),\La(\de(t)))\raa\d t +2\la X(t), \sigma(X(t),\La(t))\d W(t)\raa.
  \end{split}
  \end{equation} Then, by condition (\textbf{H1}) we get that for $t\in [0,\tau)$, \beq{0tau}
  \begin{split}
    |X(t)|^2&=|X(0)|^2+\int_{0}^{t}\big(2\la X(s),a(X(s),\La(s))\raa+\norm{\sigma(X(s),\La(s))}_{\mathrm{HS}}^2\big)\d s\\
    &\quad -2\int_{0}^{t}\la X(s),b(X(0),\La(0))\raa\d s +M(t)\\
    &\le \big(2{\hat b}^{2}t+1\big)\big(|X(0)|^{2}+1\big)
    +\big(\bar C+1\big)\int_{0}^{t} |X(s)|^{2}\d s +M(t),
  \end{split}\eeq where $M(t)$ is a continuous local martingale and $\bar C=\max_{i\in\S} C(i)$. Taking expectations on both sides yields that
  \begin{equation*}
  \E|X(t)|^2\leq \big(2{\hat b}^{2}t+1\big)\big(\E|X(0)|^{2}+1\big)+\big(\bar C+1\big)\int_0^t \E|X(s)|^{2}\d s,
  \end{equation*} and then by Gronwall's inequality, it follows that
  \begin{equation}\label{2.11A}
  \E|X(t)|^2\leq \big(2{\hat b}^{2}t+1\big)\big(\E|X(0)|^{2}+1\big)\e^{(\bar C+1) t},\quad \text{for every $t\in [0,\tau)$}.
  \end{equation} Using Fatou's lemma, from \eqref{2.11A} we also have
 \beq{2.11A2}
  \E|X(\tau)|^2\leq \big(2{\hat b}^{2}\tau+1\big)\big(\E|X(0)|^{2}+1\big)\e^{(\bar C+1)\tau}.\eeq

  Deducing inductively,  we can obtain  that for any integer $m\ge 1$,
  \begin{equation}\label{2.11A3}
  \E|X(m\tau+t)|^2\leq \big(2{\hat b}^{2}t+1\big)\big(\E|X(m\tau)|^{2}+1\big)\e^{(\bar C+1) t},\quad \text{for every $t\in [0,\tau)$},
  \end{equation} which further implies that  $\E|X(t)|^2<\infty$ for any   $t\ge 0$.
  Therefore, the explosion time $\wdt{\tau}_{\infty}$ of the component $X$ must be infinity almost surely, and so \eqref{inftyone} holds.

\emph{Third}, we go to prove that for any given $m\ge [|x|]+1$, \beq{inftyzero}\p\big(\tau_{\infty}<\wdt{\tau}_{m}\big)=0, \quad \hbox{or equivalently} \quad \p\big(\tau_{\infty}\ge\wdt{\tau}_{m}\big)=1.\eeq Indeed, for any arbitrarily fixed $m_{0}\ge [|x|]+1$, let \[\wdh{H}:=K_{0}\big(1+m_{0}^{\kappa_{0}}\big).\]
Since $|X(s)|\le m_{0}$ for all $s\le \wdt{\tau}_{m_{0}}$, so by condition (\textbf{Q2'}) we have
$q_{ij}(X(s))\le \wdh{H}$ for all $i, j\in\S$ and all $s\le \wdt{\tau}_{m_{0}}$. Then, it follows from \eqref{eq:La-SDE} that for any $t\le \tau_{\infty}$,
\beq{eq:La-SDE2}\begin{split}
 \La(t\wedge\wdt{\tau}_{m_{0}})&=\La(0) + \int_{0}^{t\wedge\wdt{\tau}_{m_{0}}}\int_{\R_{+}}h(X(s),\La(s-),z){N}(\d s,\d z)\\
 &=\La(0) + \int_{0}^{t\wedge\wdt{\tau}_{m_{0}}}\int_{[0, M(M-1)\wdh H]}h(X(s-),\La(s-),z){N}(\d s,\d z),
\end{split}\eeq since the integrand $h(X(s),\La(s-),z)$ equals $0$ when $s\le \wdt{\tau}_{m_{0}}$ and $z \notin [0, M(M-1)\wdh H]$, where constant $M$ is the number of elements in $\S$. For the characteristic measure (i.e., the intensity measure) $\m(\cdot)$ of the Poisson random measure $N(\cdot,\cdot)$, since $\m\bigl([0, M(M-1)\wdh H]\big)<\infty$, so the
stationary point process corresponding to the above Poisson random measure $N(\cdot,\cdot)$ has only finite {\it occurrence times} on the temporal interval $[0,t\wedge\wdt{\tau}_{m_{0}})$ almost surely.
Hence, it follows from \eqref{eq:La-SDE2} that the component $\La$ has only finite jumps (i.e., switches) on the temporal interval $[0,t\wedge\wdt{\tau}_{m_{0}})$ almost surely; refer to \cite[Proposition 2.1 and Corollary 2.2]{XY2011} for the details. This implies that $\tau_{\infty}\ge t\wedge\wdt{\tau}_{m_{0}}$ almost surely, and then that $\tau_{\infty}\ge \wdt{\tau}_{m_{0}}$ almost surely due to that $t$ is arbitrary. Now we actually have proven \eqref{inftyzero} since the above $m_{0}\ge [|x|]+1$ is also arbitrary.

Let
$A_{m}:=\big\{\omega\in \Omega: \tau_{\infty} < \wdt{\tau}_{m}\big\}$ for $m\ge [|x|]+1$, and let $$A=\bigcup_{m=[|x|]+1}^{\infty}A_{m}, \quad A^{c}=\Omega\backslash A .$$
It follows from \eqref{inftyzero} that
\beq{A0Ac1}\p(A)=0, \quad \hbox{and then} \quad \p(A^{c})=1.\eeq Note that $A^{c}= \{ \tau_{\infty} \ge \wdt \tau_{\infty}\}$. Thus, by  \eqref{inftyone} we have \[\p(\tau_{\infty}=\infty)\geq \p(\wdt \tau_{\infty}=\infty)=1.\]
Consequently, the solution $(X(t),\La(t))$ is nonexplosive, and so the above interlacing procedure actually determines a process $(X(t),\La(t))$ for all $t \in[0,\infty)$. The proof is thus completed.

\end{document}